\documentclass[a4paper,10pt,twoside]{amsart}
\usepackage{amsfonts, amsmath, amssymb, mathabx, appendix, mathrsfs, esint}
\usepackage{enumerate}
\usepackage{enumitem}
\usepackage{graphicx}
\numberwithin{equation}{section}
\usepackage{hyperref}

\makeatletter
\DeclareRobustCommand*{\bfseries}{%
  \not@math@alphabet\bfseries\mathbf
  \fontseries\bfdefault\selectfont
  \boldmath
}
\makeatother

\def\wto{\rightharpoonup}

\def\RR{\mathbb{R}}

\def\spt{\rm spt}
\def\mwto{\stackrel{\ast}{\wto}}
\def\MM{\mathbb{M}}
\def\Mm{\mathscr M}
\def\NN{\mathbb{N}}
\def\ZZ{\mathbb Z}

\def\LL{{\mathbb{L}}}

\def\H{{\mathcal H}}

\def\Q{Q}

\def\g{\widetilde{g}}

\DeclareMathOperator{\maxw}{{r}_{0}}

\DeclareMathOperator*{\limin}{\underline{\lim}}
\DeclareMathOperator*{\limsu}{\overline{\lim}}

\DeclareMathOperator{\ssetminus}{\setminus}

\def\ssup{\!>\!}
\def\sinf{\!<\!}

\def\M{\mathrm{m}}

\def\LL{\mathbb{L}}

\def\GG{\mathbb{G}}

\def\diam{{\rm diam}}
\def\dom{{\rm dom}}

\def\inte{{\rm int}}

\def\eps{\varepsilon}

\def\O{\mathcal O}

\def\G{\mathcal G}
\def\S{\mathcal S}

\def\t{{t_\ast}}

\def\W{{W}}
\def\sgn{\textrm{sgn}}

\DeclareMathAlphabet{\mathpzc}{OT1}{pzc}{m}{it}

\newtheorem{theorem}{Theorem}[section]
\newtheorem{lemma}{Lemma}[section]

\newtheorem{proposition}{Proposition}[section]
\newtheorem{corollary}{Corollary}[section]
\newtheorem{definition}{Definition}[section]

\renewcommand{\liminf}{\limin}
\renewcommand{\limsup}{\limsu}
\theoremstyle{remark}
\newtheorem{remark}{Remark}[section]

\newlist{hyp}{enumerate}{1}
\setlist[hyp,1]{label={\rm (${\rm A}_{\arabic*}$)}}

\newlist{hypetoile}{enumerate}{1}
\setlist[hypetoile,1]{$\Diamond$}

\newlist{hyp1}{enumerate}{1}
\setlist[hyp1,1]{label={\rm ({\roman*})}}

\newlist{hypg}{enumerate}{1}
\setlist[hypg,1]{label={\rm (${\rm H}_{\arabic*}$)}}

\newlist{hypgg}{enumerate}{1}
\setlist[hypgg,1]{label={\rm (${\rm H}_{\arabic*}^\prime$)}}

\newlist{hypc}{enumerate}{1}
\setlist[hypc,1]{label={\rm (${\rm C}_{\arabic*}$)}}

\newlist{hypgr}{enumerate}{1}
\setlist[hypgr,1]{label={\rm (${\rm D}_{\arabic*}$)}}

\newlist{hypno}{enumerate}{1}
\setlist[hypno,1]{label=}

\numberwithin{equation}{section}

\title[Homogenization with quasiconvex growth]{Homogenization of unbounded integrals with quasiconvex growth}

\author[Omar Anza Hafsa]{Omar Anza Hafsa}
\author[Jean-Philippe Mandallena]{Jean-Philippe Mandallena}
\address{Laboratoire LMGC, UMR-CNRS 5508, Place Eug\`ene Bataillon, 34095 Montpellier, France.}
\address{Universit\'e de N\^\i mes, Laboratoire MIPA, Site des Carmes, Place Gabriel P\'eri, 30021 N\^\i mes, France}
\email{Omar Anza Hafsa <omar.anza-hafsa@unimes.fr>}
\email{Jean-Philippe Mandallena <jean-philippe.mandallena@unimes.fr>}

\author[Hamdi Zorgati]{Hamdi Zorgati}
\address{Universit\'e de Tunis El Manar, Facult\'e des Sciences de Tunis, Laboratoire EDP, Tunisie}
\email{Hamdi Zorgati <hamdi.zorgati@fst.rnu.tn>}

\keywords{Homogenization, quasiconvex growth, unbounded integrals, ru-usc}
\begin{document}
\begin{abstract} We study homogenization by $\Gamma$-convergence of periodic nonconvex integrals when the integrand has  quasiconvex growth with convex effective domain. 
\end{abstract}
\maketitle
\section{Introduction} 

Let $m,d\ge 1$ be two integers and $p\in [1,\infty[$. Let $\Omega\subset\RR^d$ be a nonempty bounded open set with Lipschitz boundary. For each $\eps\ssup 0$, we define $I_\eps:W^{1,p}(\Omega;\RR^m)\to [0,\infty]$ by
\[
I_\eps(u):=\int_\Omega \W\left(\frac{x}{\eps},\nabla u(x)\right)dx,
\]
where the integrand $\W:\RR^d\times \MM^{m\times d}\to [0,\infty]$ is Borel measurable and $1$-periodic with respect to the first variable, i.e., for every $x\in\Omega$, $z\in \ZZ^d$ and $\xi\in\MM^{m\times d}$ we have
\begin{align*}
\W(x+z,\xi)=\W(x,\xi).
\end{align*}

Nonconvex homogenization by $\Gamma$-convergence of the family $\{I_\eps\}_{\eps>0}$ was mainly studied in the framework of $p$-polynomial growth conditions on $\W$. Unfortunately, this framework is not compatible with the two basic conditions of hyperelasticity: the non-interpenetration of the matter, i.e., $\W(x,\xi)=\infty$ if and only if $\det(I+\xi)\le 0$, and the necessity of an infinite amount of energy to compress a finite volume into zero volume, i.e., for every $x\in\RR^d$, $\W(x,\xi)\to\infty$ as $\det(I+\xi)\to 0$. 

At present, it seems difficult to take these conditions into account in homogenization problems. Generally, the attempts to go ``beyond" the $p$-polynomial growth are not easy due to the lack of available techniques. However, in the scalar case, we refer to the book \cite{carbone-dearcangelis02} where relaxation and homogenization of unbounded functionals were studied. 

In the vectorial case, partial investigations can be found in~\cite{oah-jpm11} where the homogenization of nonconvex integrals $\{I_\eps\}_{\eps>0}$ with convex growth conditions was carried out, and in~\cite{oah-jpm12-rm} where the homogenization in $W^{1,\infty}$ without growth conditions but with $\W$ having fixed bounded convex effective domain was studied.
 
In this paper we extend the homogenization result of~\cite{oah-jpm11}, when $p\ssup d$, to the case where $\W$ has {\em quasiconvex growth conditions} with {\em convex effective domain}. The main difficulty comes from the proof of the upper bound for the $\Gamma$-limit. Indeed, in the setting of convex growth conditions on $\W$ we can use mollifiers techniques to construct approximations of Sobolev functions by smooth ones. However, when we deal with quasiconvex growth, we need to develop other techniques. 
We will consider a set function which is a pointwise limit of local Dirichlet minimization problems associated to the family $\{I_\eps\}_{\eps>0}$ together with measure and localization arguments (introduced by~\cite{bouchitte-fonseca-mascarenhas98}) which reduce the proof of the upper bound to cut-off techniques, avoiding then any approximation arguments.

\subsection*{Outline of the paper} In Sect.~\ref{mr} we present the assumptions on $G$ and $\W$. Roughly, $\W$ is assumed to have $G$-growth conditions with $G$ independent of $x$ and satisfying a condition which avoid strong ``bumps". We establish the homogenization result Theorem~\ref{main result-cor} whose proof is based on two propositions. Proposition~\ref{main result} is concerned with the lower bound of the $\Gamma$-limit, and the upper bound of the $\Gamma$-limit in the restrictive case where the gradients belong to the interior of the effective domain. Next, we may need to extend the homogenized integrand to the boundary of the effective domain, this is the purpose of Proposition~\ref{lsc-of-HW}. At the end of the section we show that our result is an extension of the classical homogenization theorem with $p$-polynomial growth in the case $p\ssup d$.

In Sect.~\ref{preliminaries} we present some preliminary notions and results needed in the proofs of the main result. We first give an analogue property of convex functions for nonconvex integrands satisfying~\ref{A2} and~\ref{A3}. Then, we give the definition and some properties of radially uniformly upper semicontinuous integrands. In Subsect.~\ref{subtheorem}, we recall some basic facts about subadditive invariant set functions which allow easily to caracterize the homogenized formula. Subsect.~\ref{localdirichletproblems} is devoted to the introduction of the pointwise limit of local Dirichlet minimization problems associated to a family of variational functionals. 

In Sect.~\ref{proof-lsc-hw} we prove Proposition~\ref{lsc-of-HW}.

In Sect.~\ref{proof-mainresult1} we prove the lower bound for the $\Gamma$-limit by the method of localization and cut-off techniques. 

In Sect.~\ref{proof-mainresult2} we prove the upper bound for the $\Gamma$-limit for gradients in the interior of the effective domain in three steps. The first step consists in proving that the $\Gamma$-$\lim\!\sup$ is lower than a suitable envelope (similar to a Carath\'eodory type envelope in measure theory) of a set function given by the pointwise limit of local Dirichlet minimization problems associated to the family $\{I_\eps\}_{\eps>0}$. This envelope turns out to be a nonnegative finite Radon measure by a domination condition coming from the $G$-growth conditions. Then, the second step is devoted to prove the local equivalence of the envelope with the set function through Radon-Nikodym derivative. In the last step, we use cut-off functions techniques allowing to substitute the Sobolev functions with their affine tangents maps and we conclude by a subadditive argument which gives the homogenized formula.

In Sect.~\ref{proof-theorem} we prove Theorem~\ref{main result-cor} using Propositions~\ref{main result} and~\ref{lsc-of-HW}.

In Sect.~\ref{example} we give an example, when $d=m=2$, of $\W$ and $G$ satisfying the requirements of the homogenization result. Moreover, we show that the two basic conditions of hyperelasticity can be considered for special constraints on the deformations which make the tension/compression more important than shear deformations.

\section{Main result}\label{mr} 
Let $G:\MM^{m\times d}\to [0,\infty]$ be a Borel measurable function which will play the role of growth conditions on $\W$. We denote by $\GG$ the effective domain of $G$ and $\inte(\GG)$ the interior of $\GG$. We consider the following conditions on $G$:
\begin{hypc}

\item\label{A2} $0\in\inte(\GG)$;

\item\label{A3} there exists $C>0$ such that for every $\xi,\zeta\in\MM^{m\times d}$ and every $t\in ]0,1[$
\begin{align*}
G(t\xi+(1-t)\zeta)\le C(1+G(\xi)+G(\zeta));
\end{align*}

\item\label{A0} $G$ is $W^{1,p}$-quasiconvex, i.e., for every $\xi\in\MM^{m\times d}$
\begin{align*}
G(\xi)=\inf\left\{\int_{Y} G(\xi+\nabla \varphi(x))dx: \varphi\in  W^{1,p}_0(Y;\RR^m)\right\}
\end{align*}
where $Y=]0,1[^d$.\\
\end{hypc}
\begin{remark} \label{convexity-eff-domain-remark} Some comments on the previous assertions are in order:
\begin{enumerate}[label=(\roman*)]
\item The condition~\ref{A3} forbids the possible ``strong bumps" of $G$.
\item \label{convexity-eff-domain} Note that the effective domain $\GG$ is {\em convex} when~\ref{A3} holds. Moreover, if both~\ref{A2} and~\ref{A3} hold then we have the well-known property 
\begin{align*}
t\xi\in\inte(\GG)
\end{align*}
for all $t\in [0,1[$ and all $\xi\in\overline{\GG}$, where $\overline{\GG}$ denotes the closure of $\GG$. It is equivalent to
\begin{align*}
\inte(\GG)=\mathop{\cup}_{t\in[0,1[}t\overline{\GG}.
\end{align*}
\end{enumerate}
\end{remark}

The integrand $\W$ is supposed to verify the following growth and  coercivity conditions:
\begin{hypg}
\item\label{B1} $G$-growth conditions, i.e., there exist $\alpha,\beta>0$ such that for every $x\in\RR^d$ and every $\xi\in\MM^{m\times d}$
\begin{align*}
\alpha G(\xi)\le \W(x,\xi)\le\beta (1+G(\xi));
\end{align*}
\item\label{A1} $\W$ is $p$-coercive, i.e., there exists $c\ssup0$ such that for every $(x,\xi)\in\RR^d\times\MM^{m\times d}$  
\begin{align*}
c\vert\xi\vert^p\le \W(x,\xi).
\end{align*}
\end{hypg}

If $L:\Omega\times\MM^{m\times d}\to [0,\infty]$ is a Borel measurable integrand which is $1$-periodic with respect to the first variable then $\H L:\MM^{m\times d}\to [0,\infty]$ defined by 
\begin{align}\label{BM-formula}
\H L(\xi):=\inf_{k\in\NN^*}\inf\left\{\fint_{kY} L(x,\xi+\nabla \varphi(x))dx:\varphi\in W^{1,p}_0(kY;\RR^m)\right\}
\end{align}
is usually called the {\em Braides-M\"uller homogenization formula} (see~\cite{braides85,muller87}). When $L$ has $p$-polynomial growth, the corresponding homogenized functional by $\Gamma$-convergence with respect to the strong topology of $L^p(\Omega;\RR^m)$ is given by
\begin{align*}
\int_\Omega \H L(\nabla u(x))dx
\end{align*}
for all $u\in W^{1,p}(\Omega;\RR^m)$ (see Subsect.~\ref{pgrowth-homogenization}).
\begin{remark}\label{domhw} If~\ref{B1} and~\ref{A0} hold then $\dom \H\W=\GG$. Indeed, by a change of variables and periodicity arguments we see that
\begin{align*}
\inf_{\varphi\in W^{1,p}_0(nY;\RR^m)}\fint_{nY} G(\xi+\nabla \varphi(x))dx=G(\xi)
\end{align*}
for any $\xi\in\MM^{m\times d}$ and $n\in\NN$. So, using the $G$-growth conditions~\ref{B1} we obtain
\begin{align*}
\alpha G(\xi)\le\H\W(\xi)\le\beta(1+G(\xi))
\end{align*}
which implies $\dom\H\W=\GG$.
\end{remark}

Since the effective domain of $\W$ is not necessarily the whole space $\MM^{m\times d}$, we may need to extend the homogenized integrand to the boundary of $\GG$. For this purpose, we say that $L$ is {\em periodically radially uniformly upper semicontinuous (periodically ru-usc)} if there exists $a\in L^1_{\rm loc}(\RR^d;]0,\infty[)$ $1$-periodic such that\\ 
\begin{align*}
\limsup\limits_{t\to1^-}\Delta^a_L(t)\leq 0
\end{align*} 
with 
\begin{align*}
 \Delta_L^a(t):=\sup_{x\in U}\sup_{\xi\in \LL_x}{L(x,t\xi)-L(x,\xi)\over a(x)+L(x,\xi)}
\end{align*}
where $\LL_x$ is the effective domain of $L(x,\cdot)$ (see Subsect.~\ref{ru-usc functions}). 

\medskip

The homogenization of the family $\{I_\eps\}_{\eps>0}$ is achieved by $\Gamma$-convergence. We refer to the book of G. Dal Maso~\cite{dalmaso93} for a good introduction to the $\Gamma$-convergence theory. We give a brief description and specify some notation. We denote by $\O(\Omega)$ the set of all open subsets of $\Omega$. Define $I_-,I_+:W^{1,p}(\Omega;\RR^m)\times \O(\Omega)\to [0,\infty]$ by
\begin{align*}
I_-(u;O):=&\inf\left\{\liminf_{\eps\to 0}I_\eps(u_\eps;O): u_\eps\to u\mbox{ in }L^p(\Omega;\RR^m)\right\};\\
I_+(u;O):=&\inf\left\{\limsup_{\eps\to 0}I_\eps(u_\eps;O): u_\eps\to u\mbox{ in }L^p(\Omega;\RR^m)\right\}.
\end{align*}
For any $O\in \O(\Omega)$, the functional $I_-(\cdot;O)$ (resp. $I_+(\cdot;O)$) is called the $\Gamma$-$\lim\!\inf$ (resp. the $\Gamma$-limsup)  with respect to the strong topology of $L^p(\Omega;\RR^m)$ of the family $\{I_\eps(\cdot;O)\}_{\eps>0}$. Note that we always have $I_+(\cdot;O)\ge I_-(\cdot;O)$. When $I_+(\cdot;O)= I_-(\cdot;O)$ we say that the family $\{I_\eps(\cdot;O)\}_{\eps>0}$ $\Gamma$-converges with the $\Gamma$-limit given by the common value and we write 
\begin{align*}
I_0(\cdot;O)=\Gamma\mbox{-}\lim_{\eps\to 0}I_\eps(\cdot;O).
\end{align*}
When $O=\Omega$ we simply write
\begin{align*}
I_0(\cdot)=\Gamma\mbox{-}\lim_{\eps\to 0}I_\eps(\cdot).
\end{align*}

\bigskip

We define the {\em radial extension} of $\H\W$ by
\begin{align*}
\widehat{\H\W}(\xi):=\liminf\limits_{t\to 1^-}\H\W(t\xi)
\end{align*}
for all $\xi\in\MM^{m\times d}$.
\medskip

Here is the main result of our paper. 
\begin{theorem}\label{main result-cor} Assume that $p\ssup d$. Assume that~\ref{A2},~\ref{A3},~\ref{A0},~\ref{B1} and~\ref{A1} hold. If $\W$ is periodically ru-usc then $\{I_\eps\}_{\eps>0}$ $\Gamma$-converges with respect to the strong topology of $L^p(\Omega;\RR^m)$ to $I_{0}:W^{1,p}(\Omega;\RR^m)\to [0,\infty]$ given by 
\begin{align*}
I_{0}(u)=\int_\Omega \widehat{\H\W}(\nabla u(x))dx.
\end{align*}
\end{theorem}

Define $\G, \H I, \widehat{\H I}:W^{1,p}(\Omega;\RR^m)\times\O(\Omega)\to [0,\infty]$ by
\begin{enumerate}
\item[$\blacklozenge$] $\displaystyle\G(u;O):=\int_O G(\nabla u(x))dx;$\\

\item[$\blacklozenge$] $\displaystyle\H I(u;O):=\int_O\H\W(\nabla u(x))dx;$\\

\item[$\blacklozenge$] $\displaystyle\widehat{\H I}(u;O):=\int_O\widehat{\H\W}(\nabla u(x))dx$.
\end{enumerate}

The proof of Theorem~\ref{main result-cor} is based on the following result.

\begin{proposition}\label{main result} Assume that $p\ssup d$. Assume that~\ref{A2},~\ref{A3},~\ref{B1} and~\ref{A1} hold. Let $u\in W^{1,p}(\Omega;\RR^m)$ and $O\in\O(\Omega)$. 
\begin{enumerate}[label=(\roman*)]
\item\label{mr1} If $\W$ is periodically ru-usc then 
\begin{align*}
I_-(u;O)\ge \widehat{\H I}(u;O).
\end{align*}
\item\label{mr2} Let $t\in ]0,1[$. If $\G(tu;O)\sinf\infty$ and if there exists $\t\in ]t,1[$ such that $\G(\t u;O)\sinf\infty$ then
\begin{align*}
I_+(tu;O)\le \H I(tu;O).
\end{align*}
\end{enumerate}
\end{proposition}
The following proposition gives some properties of $\H\W$ and $\widehat{\H\W}$ when $\W$ is ru-usc. 
\begin{proposition}\label{lsc-of-HW} Assume that~\ref{A2},~\ref{A3},~\ref{A0} and~\ref{B1} hold. If $\W$ is periodically ru-usc then $\H\W$ and $\widehat{\H\W}$ are ru-usc. Moreover, we have 
\[
\widehat{\H\W}(\xi)=\lim_{t\to 1^-}\H\W(t\xi)
\]
for all $\xi\in\MM^{m\times d}$.
\end{proposition}
\begin{remark}\label{lsc=widehat} Under the assumptions of Theorem~\ref{main result-cor} it holds that
\begin{align*}
\widehat{\H\W}= \overline{\H\W}
\end{align*}
where the bar denotes the lower semicontinuous envelope of $\H\W$. Indeed, first it is easy to see that $\widehat{\H\W}\ge \overline{\H\W}$. On the other hand, $\widehat{\H\W}$ is ru-usc by Proposition~\ref{lsc-of-HW}. By Remark~\ref{ReMaRk-Delta-ru-usc-remARK1} we have $\widehat{\H\W}\leq {\H\W}$ on $\dom\H\W$ and $\dom\H\W=\GG$ by Remark~\ref{domhw}, so it follows that $\overline{\widehat{\H\W}}\le \overline{\H\W}$. But Theorem~\ref{main result-cor} implies that ${\widehat{\H\W}}$ is lower semicontinuous, i.e., $\widehat{\H\W}=\overline{\widehat{\H\W}}$, and then $\widehat{\H\W}\le \overline{\H\W}$.
\end{remark}

\subsection{Application to homogenization with $p$-polynomial growth}\label{pgrowth-homogenization} We want to show how to recover, from Theorem~\ref{main result-cor}, the classical homogenization result with $p$-polynomial growth on the integrand for $p\ssup d$ (see~\cite{braides85,muller87}).

Let $L:\RR^d\times\MM^{m\times d}\to [0,\infty]$ be a Borel measurable, $1$-periodic with respect the first variable, and satisfying $p$-polynomial growth: there exist $\alpha,\beta\ssup 0$ such that for every $\xi\in\MM^{m\times d}$
\begin{align}\label{pgrowthf}
\alpha \lvert\xi\rvert^p\le L(x,\xi)\le\beta (1+\lvert\xi\rvert^p).
\end{align}
For each $\eps\ssup0$ we define $J_\eps:W^{1,p}(\Omega;\RR^m)\to [0,\infty]$ by 
\begin{align*}
J_\eps(u):=\int_\Omega L\left(\frac{x}{\eps},\nabla u(x)\right)dx.
\end{align*}
\begin{theorem}\label{classic-homogenization} Let $p\ssup d$. The family $\{J_\eps\}_{\eps>0}$ $\Gamma$-converges with respect to the strong topology of $L^{p}(\Omega;\RR^m)$ to $J_0:W^{1,p}(\Omega;\RR^m)\to [0,\infty]$ given by
\begin{align}\label{J-gammalimit}
J_0(u):=\int_\Omega \H L(\nabla u(x))dx.
\end{align}
\end{theorem}
\begin{proof}First, we apply our result to the family $\{I_\eps\}_{\eps>0}$ with the growth given by $G(\cdot):=\lvert\cdot\rvert^p$ and the corresponding integrand given by $\W$ the {\em quasiconvexification} of $L$, i.e., 
\begin{align*}
\W(x,\xi)=\Q L(x,\xi):=\inf\left\{\int_YL(x,\xi+\nabla \varphi(y))dy:\varphi\in W^{1,\infty}_0(Y;\RR^m)\right\}
\end{align*}
for all $(x,\xi)\in \RR^d\times \MM^{m\times d}$. 

\medskip

It is easy to see that~\ref{A2},~\ref{A3},~\ref{A0},~\ref{B1} and~\ref{A1} are satisfied since \eqref{pgrowthf}. It remains to verify that $\W(=\Q L)$ is periodically ru-usc. Fix any $t\in[0,1]$, any $x\in\RR^d$ and any $\xi\in\GG$. As $\W$ is quasiconvex and satisfies \eqref{pgrowthf}, there exists $K>0$ such that
\begin{equation}\label{Intro-Bis-Lipschitz}
|\W(x,\zeta)-\W(x,\zeta^\prime)|\leq K|\zeta-\zeta^\prime|(1+|\zeta|^{p-1}+|\zeta^\prime|^{p-1})
\end{equation}
for all $x\in\RR^d$ and all $\zeta,\zeta^\prime\in\MM^{d\times d}$. Using \eqref{Intro-Bis-Lipschitz} with $\zeta=t\xi$ and $\zeta^\prime=\xi$ and taking the left inequality in \eqref{pgrowthf} into account, we obtain
\begin{equation}\label{Intro-Eq-Prop-Eq1}
\W(x,t\xi)-\W(x,\xi)\leq K^\prime(1-t)(1+\W(x,\xi))
\end{equation}
with $K^\prime:=3K\max\{1,{1\over \alpha}\}$. Dividing by $1+\W(x,\xi)$ and passing to the supremum in $x\in\RR^d$ and $\xi\in\MM^{m\times d}$, we have
\begin{align*}
\Delta_\W^1(t)\le K^\prime(1-t).
\end{align*}
Passing to the limit $t\to 1$, we have $\W$ is periodically ru-usc. 

Applying Theorem~\ref{main result-cor} we have for every $u\in W^{1,p}(\Omega;\RR^m)$
\begin{align}\label{gammahqf}
I_0(u):=\int_\Omega \widehat{\H(\Q L)}(\nabla u(x))dx.
\end{align}
By Remark~\ref{lsc=widehat}, we have $\widehat{\H(\Q L)}=\overline{\H(\Q L)}$. Using Dacorogna-Acerbi-Fusco relaxation result~\cite{acerbi-fusco84, dacorogna82} (see~\eqref{DAF}) we have $\H(\Q L)=\H L$ since \eqref{pgrowthf}. Assume for the moment that $\H(\Q L)$ is lower semicontinuous, i.e.,
\begin{align}\label{hqf is lsc}
\H(\Q L)=\overline{\H(\Q L)}
\end{align}
then we have
\begin{align}\label{hf=hathqf}
\H L=\widehat{\H(\Q L)}.
\end{align}
For each $\eps\ssup0$ we apply the Dacorogna-Acerbi-Fusco relaxation theorem to have 
\[
\overline{J_\eps}(u)=\int_\Omega \Q L\left(\frac{x}{\eps},\nabla u(x)\right)dx=I_\eps(u)
\]
for all $u\in W^{1,p}(\Omega;\RR^m)$, where the bar over $J_\eps$ denotes the lower semicontinuous envelope with respect to the strong topology of $L^p(\Omega;\RR^m)$. On the other hand, it is well-known that 
\begin{align}\label{gamma-relax}
(J_0(u)=)\Gamma\hbox{\rm -}\lim_{\eps\to 0}J_\eps(u)=\Gamma\hbox{\rm -}\lim_{\eps\to 0}\overline{J_\eps}(u)\left(=I_0(u)\right)
\end{align}
for all $u\in W^{1,p}(\Omega;\RR^m)$. Collecting \eqref{gammahqf}, \eqref{hf=hathqf} and \eqref{gamma-relax}, we finally obtain \eqref{J-gammalimit}.

\medskip

The only thing remains to prove is \eqref{hqf is lsc}. Let $\{\xi_\eps\}_{\eps>0},\xi_0\in \MM^{m\times d}$ be such that $\xi_\eps\to\xi_0$ as $\eps\to 0$. Without loss of generality, we can assume that 
\begin{align*}
\liminf_{\eps\to 0}\H(\Q L)(\xi_\eps)=\lim_{\eps\to 0}\H(\Q L)(\xi_\eps)\sinf\infty,\;\mbox{ and }M:=\sup_{\eps>0}\H(\Q L)(\xi_\eps)\sinf\infty.
\end{align*}
Fix $\eps\in ]0,1[$. We choose $k_\eps\in\NN$ and $\varphi_\eps\in W^{1,p}(k_\eps Y;\RR^m)$ such that
\begin{align*}
M+1\ge \eps+\H(\Q L)(\xi_\eps)&\ge \fint_{k_\eps Y} \Q L(x,\xi_\eps+\nabla \varphi_\eps(x))dx\\&=\int_{Y} \Q L(k_\eps y,\xi_\eps+\nabla \varphi_\eps(k_\eps y))dy,
\end{align*}
where a change of variable is used. Set $\phi_\eps(\cdot):=\frac{1}{k_\eps}\varphi(k_\eps\cdot)$. It is easy to see that $\Q L$ satisfies~\eqref{pgrowthf} with the same constants. So, it follows that
\begin{align}\label{bound-coercivity}
\sup_{\eps>0} \int_Y \lvert \xi_\eps+\nabla\phi_\eps(x)\rvert^pdx\le \frac{M+1}{c}.
\end{align}
By H\"older inequality and \eqref{bound-coercivity} we have for every $\eps\ssup0$
\begin{align}\label{holderqc}
\left\Vert \xi_\eps+\nabla\phi_\eps\right\Vert_{L^{p-1}(Y;\RR^m)}\le\left\Vert \xi_\eps+\nabla\phi_\eps\right\Vert_{L^{p}(Y;\RR^m)} \le\left(\frac{M+1}{c}\right)^{\frac1p},
\end{align}
and also by norm inequality 
\begin{align}\label{normqc}
\left\Vert \xi_0+\nabla\phi_\eps\right\Vert_{L^{p-1}(Y;\RR^m)}\le& \left\vert\xi_0-\xi_\eps\right\vert+\left\Vert \xi_\eps+\nabla\phi_\eps\right\Vert_{L^{p-1}(Y;\RR^m)}\\
\le&\left\vert\xi_0-\xi_\eps\right\vert+\left(\frac{M+1}{c}\right)^{\frac1p}.\notag
\end{align}
By the definition of $\H(\Q L)(\xi_0)$,~\eqref{Intro-Bis-Lipschitz}, \eqref{holderqc} and \eqref{normqc}, we have for every $\eps\ssup 0$
\begin{align*}
&\eps+\H(\Q L)(\xi_\eps)\\
\ge&\int_{Y} \Q L(k_\eps y,\xi_\eps+\nabla \phi_\eps(y))-\Q L(k_\eps y,\xi_0+\nabla \phi_\eps(y))dy+\H(\Q L)(\xi_0)\\
\ge&-K\lvert\xi_\eps-\xi_0\rvert\left(1+\int_Y \lvert \xi_\eps+\nabla\phi_\eps(x)\rvert^{p-1}+\lvert \xi_0+\nabla\phi_\eps(x)\rvert^{p-1}dx\right)+\H(\Q L)(\xi_0)\\
\ge&-K\lvert\xi_\eps-\xi_0\rvert\left(1+\left(\frac{M+1}{c}\right)^{\frac{p-1}{p}}\!\!+\!\!\left(\lvert\xi_0-\xi_\eps\rvert+\left(\frac{M+1}{c}\right)^{\frac{1}{p}}\right)^{{p-1}}\right)+\H(\Q L)(\xi_0),
\end{align*}
letting $\eps\to 0$ we obtain the lower semicontinuity of $\H(\Q L)$.
\end{proof}
\section{Preliminaries}\label{preliminaries}
\subsection{Consequence of assumptions~\ref{A2} and~\ref{A3}}The following lemma is an extension, for nonconvex functions satisfying~\ref{A2} and~\ref{A3}, of the classical local upper bound property for convex functions. 
\begin{lemma}\label{lemma finite integrand} Let $L:\MM^{m\times d}\to [0,\infty]$ be a Borel measurable integrand. If $L$ satisfies \ref{A2} and \ref{A3} then there exists $\rho_0\ssup0$ such that
\begin{align*}
\maxw:=\sup_{\zeta\in\overline{B}_{\rho_0}(0)}L(\zeta)\sinf\infty.
\end{align*}
\end{lemma}
\begin{proof} From \ref{A2} there exists $\rho_0\ssup0$ such that $L(\xi)\sinf\infty$ for all $\xi\in \overline{B}_{\rho_0}(0)$. Each matrix $\xi\in\overline{B}_{\rho_0}(0)$ is identified to the vector
\begin{align*}
\xi=\left(\xi_{11},\cdots,\xi_{1d},\cdots,\xi_{i1},\cdots,\xi_{id},\cdots,\xi_{m1},\cdots,\xi_{md}\right).
\end{align*}
Consider the finite subset \[\S:=\left\{(\xi_{11},\cdots,\xi_{md})\in\MM^{m\times d}: \xi_{ij}\in\{-\rho_0,0,\rho_0\}\right\}\subset \overline{B}_{\rho_0}(0)\] and we define 
$
L^\ast:=\max\limits_{\xi\in\S}L(\xi)\sinf\infty.
$ 

Let 
$
\zeta=\left(\zeta_{11},\cdots,\zeta_{1d},\cdots,\zeta_{i1},\cdots,\zeta_{id},\cdots,\zeta_{m1},\cdots,\zeta_{md}\right)\in\S
$
 and $\xi\in \overline{B}_{\rho_0}(0)$ with $\xi_{ij}=\zeta_{ij}$ for all $i\not=1$ and $j\not=1$. If $\xi_{11}\not=0$ then by \ref{A3} we have  
\begin{align}\label{Eq1: lemma finite integrand}
L(\xi)=&L\left(\mbox{$\frac{\vert\xi_{11}\vert}{\rho_0}$}\sgn(\xi_{11})\rho_0+\left(1-\mbox{$\frac{\vert\xi_{11}\vert}{\rho_0}$}\right)0,\cdots,\xi_{1d},\cdots,\xi_{m1},\cdots,\xi_{md}\right)\\
\le& C\left(1+L(\rho_0,\cdots,\xi_{md})+L(0,\cdots,\xi_{md})\right)\notag\\
\le& 2C\left(1+L^\ast\right)\notag
\end{align}
where $\sgn(\xi_{ij})$ denotes the sign of $\xi_{ij}$. The same upper bound in \eqref{Eq1: lemma finite integrand} holds for $L(\xi)$ when $\xi_{11}=0$. 

Assume now that $\xi_{ij}=\zeta_{ij}$ for all $i\not=1$ and $j\notin\{1,2\}$. Then using \eqref{Eq1: lemma finite integrand} and \ref{A3}, we have
\begin{align*}
L(\xi)=&L\left(\xi_{11},\mbox{$\frac{\vert\xi_{12}\vert}{\rho_0}$}\sgn(\xi_{12})\rho_0+\left(1-\mbox{$\frac{\vert\xi_{12}\vert}{\rho_0}$}\right)0,\cdots,\xi_{1d},\cdots,\xi_{m1},\cdots,\xi_{md}\right)\\
\le& C\left(1+2C(1+L^\ast)+L^\ast\right)\notag\\
\le& C(1+2C)\left(1+L^\ast\right).\notag
\end{align*}
Recursively, we obtain $C^\ast\ssup0$ which depends on $C$ only, such that 
\[
L(\xi)\le C^\ast(1+L^\ast)
\] 
for all $\xi\in\overline{B}_{\rho_0}(0)$. 
\end{proof}
\subsection{Ru-usc functions}\label{ru-usc functions} Let $U\subset\RR^d$ be an open set and let $L:U\times\MM^{m\times d}\to[0,\infty]$ be a Borel measurable function. For each $x\in U$, we denote the effective domain of $L(x,\cdot)$
 by $\LL_x$ and, for each $a\in L^1_{\rm loc}(U;]0,\infty[)$, we define $\Delta_L^a:[0,1]\to]-\infty,\infty]$ by
\[
 \Delta_L^a(t):=\sup_{x\in U}\sup_{\xi\in \LL_x}{L(x,t\xi)-L(x,\xi)\over a(x)+L(x,\xi)}.
 \]
 \begin{definition}\label{ru-usc-Def}
 We say that $L$ is {\em radially uniformly upper semicontinuous (ru-usc)} if there exists $a\in L^1_{\rm loc}(U;]0,\infty[)$ such that
 \[
 \limsup_{t\to 1}\Delta^a_L(t)\leq 0.
 \]
 If moreover $a$ is $1$-periodic then we say that $L$ is {\em periodically ru-usc}.
 \end{definition}
 For a detailed study of ru-usc functions see~\cite{oah-jpm12-radial}.
 \begin{remark}\label{ReMaRk-Delta-ru-usc-remARK1}
 If $L$ is ru-usc then 
 \begin{equation}\label{Delta-ru-usc-remARK1}
 \limsup_{t\to1^-}L(x,t\xi)\leq L(x,\xi)
 \end{equation}
 for all $x\in U$ and all $\xi\in\LL_x$. Indeed, given $x\in U$ and $\xi\in\LL_x$, we have
\[
 L(x,t\xi)\leq\Delta^a_L(t)\left(a(x)+L(x,\xi)\right)+L(x,\xi)\hbox{ for all }t\in[0,1],
\]
 which gives \eqref{Delta-ru-usc-remARK1} since $a(x)+L(x,\xi)>0$ and $\limsup_{t\to 1}\Delta^a_L(t)\leq 0$.
 \end{remark}
 
Define $\widehat{L}:U\times\MM^{m\times d}\to[0,\infty]$ by 
\[
\widehat{L}(x,\xi):=\liminf_{t\to 1^-}L(x,t\xi).
\]

The following lemma gives some properties of $\widehat{L}$ when $L$ is ru-usc (for the proof see also~\cite[Lemma~3.4 and Theorem~3.5~(ii)]{oah-jpm11}). 

\begin{lemma}\label{Extension-Result-for-ru-usc-Functions}
If $L$ is ru-usc and if for every $x\in U$, 
\begin{equation}\label{Homothecie-Assumption-Bis}
t\overline{\LL}_x\subset\LL_x\mbox{ for all }t\in]0,1[
\end{equation}

then
\begin{enumerate}[label=(\roman*)]
\item\label{Extension-Result-for-ru-usc-Functions1} for every $\xi\in\overline{\LL}_x$ it holds $\widehat{L}(x,\xi)=\lim\limits_{t\to 1^-}{L}(x,t\xi)$ for all $x\in U$;

\item\label{Extension-Result-for-ru-usc-Functions2}  $\widehat{L}$ is ru-usc.
\end{enumerate}
\end{lemma}
\begin{proof} First we prove~\ref{Extension-Result-for-ru-usc-Functions1}. Fix $x\in U$. We have to prove that for every $\xi\in \overline{\LL}_x$
\begin{align*}
\liminf_{t\to 1^-}L(x,t\xi)=\limsup_{t\to 1^-}L(x,t\xi).
\end{align*}

Fix $\xi\in \overline{\LL}_x$. It suffices to prove that
\begin{equation}\label{ru-usc-AssUmPtIoN1}
\limsup_{t\to 1}\Psi(t)\leq\liminf_{t\to 1}\Psi(t).
\end{equation}
where $\Psi(t):=L(x,t\xi)$ for all $t\in [0,1]$. Without loss of generality we can assume that $\liminf_{t\to 1}\Psi(t)\sinf\infty$. Choose two sequences $\{t_n\}_n, \{s_n\}_n\subset]0,1[$ such that $t_n\to 1$, $s_n\to 1$, $\frac{t_n}{s_n}< 1$ for all $n\in\NN$, and 
\begin{align*}
\limsup_{t\to 1}\Psi(t)&=\lim_{n\to\infty}\Psi(t_n);\\
\liminf_{t\to 1}\Psi(t)&=\lim_{n\to\infty}\Psi(s_n).
\end{align*}
It is possible because, once the sequences $\{t_n\}_n, \{s_n\}_n\subset]0,1[$ satisfying $t_n\to 1$, $s_n\to 1$ choosen, we can extract a subsequence $\{s_{\sigma(n)}\}_n$ such that $\frac{t_n}{s_{\sigma(n)}}< 1$ for all $n\in\NN$. Indeed, it suffices to consider the increasing map $\sigma:\NN\to\NN$ defined by $\sigma(0):=\min\{\nu\in\NN:s_\nu>t_0\}$ and $\sigma(n+1):=\min\{\nu\in\NN:\nu>\sigma(n)\mbox{ and }s_\nu>t_{n+1}\}$ for all $n\in\NN$.

Since \eqref{Homothecie-Assumption-Bis} we have $t_n\xi\in \LL_x$ for all $n\in\NN$, so we can assert that for every $n\in\NN$
\begin{align}\label{ru-usc-AssUmPtIoN2}
\Psi(t_n)=\Psi\left(\frac{t_n}{s_n}s_n\right)&=L\left(x,\frac{t_n}{s_n} (s_n\xi)\right)\\
&\leq \Delta^a_{L}\left(\frac{t_n}{s_n}\right)(a(x)+\Psi(s_n))+\Psi(s_n).\notag
\end{align}
Letting $n\to\infty$ we deduce \eqref{ru-usc-AssUmPtIoN1} from \eqref{ru-usc-AssUmPtIoN2} since $L$ is ru-usc.

It remains to prove~\ref{Extension-Result-for-ru-usc-Functions2}, i.e., $\displaystyle\widehat{L}$ is ru-usc. Fix $t\in[0,1[$ and $\xi\in\overline{\LL}_x$. By ~\ref{Extension-Result-for-ru-usc-Functions1} we can assert that
\begin{align*}
\widehat{L}(x,\xi)&=\lim_{s\to 1}L(x,s\xi)\\
\widehat{L}(x,t\xi)&=\lim_{s\to 1}L(x,s(t\xi)).
\end{align*}
So, we have
\begin{align*}
\begin{split}
\frac{\widehat{L}(x,t\xi)-\widehat{L}(x,\xi)}{a(x)+\widehat{L}(x,\xi)}=\lim_{s\to 1}\frac{L(x,t(s\xi))-L(x,s\xi)}{a(x)+L(x,s\xi)}\le \Delta_{L}^a(t).
\end{split}
\end{align*}
It follows that $\Delta_{\widehat{L}}^a(t)\le  \Delta_{L}^a(t)$. Letting $t\to 1$, we finish the proof. 
\end{proof}
Assume that $U=\RR^d$ and define $\mathcal{H}L:\RR^d\times\MM^{m\times d}\to[0,\infty]$ by 
\[
\mathcal{H}L(\xi):=\inf_{k\geq 1}\inf\left\{\fint_{kY}L(x,\xi+\nabla\phi(x))dx:\phi\in W^{1,p}_0(kY;\RR^m)\right\}.
\]
The following result shows that the ru-usc property is stable by homogenization.

\begin{proposition}\label{Coro-Prop-ru-usc2}
If $L$ is periodically ru-usc then $\mathcal{H} L$ is ru-usc.
\end{proposition}
\begin{proof}
Fix any $t\in[0,1]$ and any $\xi\in\mathcal{H}\LL$, where $\mathcal{H}\LL$ denotes the effective domain of $\mathcal{H}L$. By definition, there exists $\{k_n;\phi_n\}_n$ such that
\begin{enumerate}
\item[$\blacklozenge$] $\displaystyle\phi_n\in W^{1,p}_0(k_nY;\RR^m)\mbox{ for all }n\geq 1$;\\
\item[$\blacklozenge$] $\displaystyle\displaystyle\mathcal{H}L(\xi)=\lim_{n\to\infty}\fint_{k_nY}L(x,\xi+\nabla\phi_n(x))dx$;\\
\item[$\blacklozenge$] $\displaystyle\xi+\nabla\phi_n(x)\in \LL_x\mbox{ for all }n\geq 1\mbox{ and a.a. }x\in k_nY$.
\end{enumerate}
Moreover, for every $n\geq 1$,
\[
\mathcal{H}L(t\xi)\leq\fint_{k_n Y}L(x,t(\xi+\nabla\phi_n(x)))dx
\]
since $t\phi_n\in W^{1,p}_0(k_nY;\RR^m)$, and so
\[
\mathcal{H}L(t\xi)-\mathcal{H}L(\xi)\leq\liminf_{n\to\infty}\fint_{k_nY}
\big(L(x,t(\xi+\nabla\phi_n(x)))-L(x,\xi+\nabla\phi_n(x))\big)dx.
\]
As $L$ is periodically ru-usc it follows that
\[
\mathcal{H}L(t\xi)-\mathcal{H}L(\xi)\leq\Delta^a_L(t)\big(\langle a\rangle+\mathcal{H}L(\xi)\big)
\]
with $\langle a\rangle:=\int_Ya(y)dy$, which implies that $\Delta^{\langle a\rangle}_{\mathcal{H}L}(t)\leq\Delta^a_L(t)$ for all $t\in[0,1]$, and the proof is complete.
\end{proof}

\subsection{Subadditive theorem}\label{subtheorem}
Let $\mathcal{O}_b(\RR^d)$ be the class of all bounded open subsets of $\RR^d$. 

\begin{definition}
Let $\mathcal{S}:\mathcal{O}_b(\RR^d)\to[0,\infty]$ be a set function.
\begin{enumerate}[label=(\roman*)]
\item We say that $\mathcal{S}$ is subadditive if 
\[
\mathcal{S}(A)\leq \mathcal{S}(B)+\mathcal{S}(C)
\]
for all $A,B,C\in\mathcal{O}_b(\RR^d)$ with $B,C\subset A$, $B\cap C=\emptyset$ and $|A\ssetminus B\cup C|=0$.
\item We say that $\mathcal{S}$ is $\ZZ^d$-invariant if
\[
\mathcal{S}(A+z)=\mathcal{S}(A)
\]
for all $A\in\mathcal{O}_b(\RR^d)$ and all $z\in\ZZ^d$.
\end{enumerate}
\end{definition}
Let ${\rm Cub}(\RR^d)$ be the class of all open cubes in $\RR^d$ and let $Y:=]0,1[^d$. The following theorem is due to Akcoglu and Krengel (see~\cite{akcoglu-krengel81}, see also~\cite{licht-michaille02}, and~\cite[Theorem~3.11]{oah-jpm11}).
\begin{theorem}\label{AK-SubadditiveTheorem}
Let $\mathcal{S}:\mathcal{O}_b(\RR^d)\to[0,\infty]$ be a subadditive and $\ZZ^d$-invariant set function for which there exists $c>0$ such that
\begin{equation}\label{Subbaditive-Hypothesis}
\mathcal{S}(A)\leq c|A|
\end{equation}
for all $A\in\mathcal{O}_b(\RR^d)$. Then, for every $Q\in{\rm Cub}(\RR^d)$,
\[
\lim_{\eps\to0}{\mathcal{S}\left({1\over\eps}Q\right)\over\left|{1\over\eps}Q\right|}=\inf_{k\geq 1}{S(kY)\over k^d}.
\]
\end{theorem}
Given a Borel measurable function $W:\RR^d\times\MM^{m\times d}\to[0,\infty]$, we define for each $\xi\in\MM^{m\times d}$, $\mathcal{S}_\xi:\mathcal{O}_b(\RR^d)\to[0,\infty]$ by 
\begin{equation}\label{SuBaDDiTiVe-W}
\mathcal{S}_\xi(A):=\inf\left\{\int_AW(x,\xi+\nabla\phi(x))dx:\phi\in W^{1,p}_0(A;\RR^m)\right\}.
\end{equation}
It is easy to see that the set function $\mathcal{S}_\xi$ is subadditive. If we assume that $\W$ is $1$-periodic with respect to the first variable, then $\mathcal{S}_\xi$ is $\ZZ^d$-invariant. Moreover, if $\W$ is such that there exist a Borel measurable function $G:\MM^{m\times d}\to[0,\infty]$ and $\beta>0$ satisfying
\begin{equation}\label{HypotheSis-SubAdditiVe-Particular}
W(x,\xi)\leq\beta(1+G(\xi))
\end{equation}
for all $\xi\in\MM^{m\times d}$, then
\[
\mathcal{S}_\xi(A)\leq\beta(1+G(\xi))|A|
\]
for all $A\in\mathcal{O}_b(\RR^d)$. 

From the above, we see that the following result is a direct consequence of Theorem~\ref{AK-SubadditiveTheorem}.
\begin{corollary}\label{SubadditiveTheorem}
Assume that $\W$ is $1$-periodic with respect to the first variable and satisfies \eqref{HypotheSis-SubAdditiVe-Particular}. Then, for every $\xi\in\GG$
\[
\lim_{\eps\to 0}{\mathcal{S}_\xi\left({1\over\eps}Q\right)\over\left|{1\over\eps}Q\right|}=\inf_{k\geq 1}{\mathcal{S}_\xi(kY)\over k^d}.
\]
\end{corollary}


\subsection{Local Dirichlet problems associated to a family of functionals}\label{localdirichletproblems}
For any family of (variational) functionals $\{H_\delta\}_{\delta>0}, H_\delta:W^{1,p}(\Omega;\RR^m)\times\O(\Omega)\to [0,\infty]$ we set
\begin{enumerate}
\item[$\blacklozenge$] $\displaystyle\M_\delta(u;O):=\inf\left\{H_\delta(v;O):v\in u+{W}^{1,p}_0(O;\RR^m)\right\}$;\\

\item[$\blacklozenge$] $\displaystyle\Mm(u;O):=\limsup_{\delta\to 0}\M_\delta(u;O)$,
\end{enumerate}
where $v\in u+{W}^{1,p}_0(O;\RR^m)$ means that $v\in W^{1,p}(\Omega;\RR^m)$ and $v-u=0$ in $\Omega\ssetminus O$ (this definition is equivalent to the classical definition of $u+{W}^{1,p}_0(O;\RR^m)$, see for instance~\cite[Chap. 9, p. 233]{hedberg-adams}). It is easy to see that we may also write $\M_\delta(u;O)=\inf\left\{H_\delta(u+\varphi;O):\varphi\in {W}^{1,p}_0(O;\RR^m)\right\}$ for all $\delta\ssup0$ and all $u\in W^{1,p}(\Omega;\RR^m)$.

\medskip

For each $\eps\ssup 0$ and each $O\in\O(\Omega)$, denote by $\mathcal{V}_\eps(O)$ the class of all countable family $\{\overline{\Q}_i:=\overline{\Q}_{\rho_i}(x_i)\}_{i\in I}$ of disjointed (pairwise disjoint) closed balls of $O$ with $x_i\in O$ and $\rho_i=\diam(\Q_i)\in ]0,\eps[$ such that $\left\vert O\ssetminus \mathop{\cup}_{i\in I}Q_i\right\vert=0$. Let $u\in W^{1,p}(\Omega;\RR^m)$ and $\eps\ssup0$. Consider $\Mm^\eps(u;\cdot):\O(\Omega)\to [0,\infty]$ defined by 
\begin{align*}
\Mm^\eps(u;O):=\inf\left\{\sum_{i\in I}\Mm(u;Q_i): \{\overline{\Q}_i\}_{i\in I}\in \mathcal{V}_\eps(O)\right\},
\end{align*}
and define $\Mm^\ast(u;\cdot):\O(\Omega)\to [0,\infty]$ by
\begin{align*}
\Mm^\ast(u;O):=\sup_{\eps> 0}\Mm^\eps(u;O)=\lim_{\eps\to 0}\Mm^\eps(u;O).
\end{align*}

\begin{lemma}\label{mastinquelity} Assume that for each $\delta>0$ and each $v\in W^{1,p}(\Omega;\RR^m)$ the set function $H_\delta(v;\cdot)$ is countably subadditive. Let $(u,O)\in W^{1,p}(\Omega;\RR^m)\times \O(\Omega)$ satisfying 
\begin{align}\label{eq00: localsup}
\sum_{i\in I}\sup_{\delta>0}H_\delta(u;\Q_i)\sinf\infty
\end{align} 
for all disjointed closed balls $\{\overline{\Q}_i\}_{i\in I}$ of $O$ satisfying $\left\vert O\ssetminus \mathop{\cup}_{i\in I}Q_i\right\vert=0$.
Then we have
\begin{align}\label{eq01: localsup}
\Mm(u;O)\le\Mm^\ast(u;O).
\end{align}
\end{lemma}
\begin{proof} Fix $(u,O)\in W^{1,p}(\Omega;\RR^m)\times\O(\Omega)$ satisfying~\eqref{eq00: localsup}. Fix $\eps\ssup0$. Choose $\{\overline{\Q}_i\}_{i\ge 1}\in \mathcal{V}_\eps(O)$ such that
\begin{align}\label{eq1:localin}
\sum_{i\ge 1} \Mm(u;\Q_i)\le \Mm^\eps(u;O)+\frac\eps2\le\Mm^\ast(u;O)+ \frac\eps2.
\end{align}
Fix $\delta\ssup0$. For each $i\ge 1$ there exists $\varphi_i\in {W}^{1,p}_0(\Q_i;\RR^m)$ such that
\begin{align}\label{eq2:localin}
H_\delta(u+\varphi_i;\Q_i)\le \frac{\delta}{2^{i+1}}+\M_\delta(u;\Q_i).
\end{align}
Set $\varphi_{\delta,\eps}:=\sum\limits_{i\ge 1}\varphi_i\mathbb{I}_{\Q_i}\in  {W}^{1,p}_0(O;\RR^m)$. By the countable subadditivity of $H_\delta(u+\varphi_{\delta,\eps};\cdot)$ and \eqref{eq2:localin} we have
\begin{align}\label{eq03: localsup}
\Mm(u;O)&\le \limsup_{\delta\to 0}H_\delta(u+\varphi_{\delta,\eps};O)\le \limsup_{\delta\to 0}\sum_{i\ge 1} H_\delta(u+\varphi_i;\Q_i)\le \limsup_{\delta\to 0}\sum_{i\ge 1} \M_\delta(u;\Q_i).
\end{align}
But, for every $\delta\ssup 0$ and every $i\ge 1$ it holds
\begin{align}\label{eq04p: localsup}
\sup_{\eta\in ]0,\delta[}\M_\eta(u;\Q_i)\le \sup_{\eta>0}H_\eta(u;\Q_i).
\end{align}
Applying the dominated convergence theorem and using~\eqref{eq00: localsup} together with~\eqref{eq04p: localsup}, we have
\begin{align}\label{eq04: localsup}
\limsup_{\delta\to 0}\sum_{i\ge 1} \M_\delta(u;\Q_i)\le\sum_{i\ge 1} \limsup_{\delta\to 0}\M_\delta(u;\Q_i)=\sum_{i\ge 1} \Mm(u;\Q_i).
\end{align}
Collecting~\eqref{eq1:localin},~\eqref{eq04: localsup} and \eqref{eq03: localsup} and letting $\eps\to 0$ we obtain \eqref{eq01: localsup}.
\end{proof}
\begin{remark}\label{condition-domination-mstar}
Let $(u,O)\in W^{1,p}(\Omega;\RR^m)\times \O(\Omega)$ and let $\G:W^{1,p}(\Omega;\RR^m)\times \O(\Omega)\to [0,\infty]$ be such that $\G(u,O)\sinf\infty$ and $\G(v,\cdot)$ is a measure for all $v\in W^{1,p}(\Omega;\RR^m)$. If there exists $\beta\ssup0$ such that for every $\delta\ssup 0$ it holds
\begin{align*}
H_\delta(u;U)\le\beta\left(\vert U\vert+\G(u;U)\right)
\end{align*}
for all $U\in \O(O)$, then \eqref{eq00: localsup} is satisfied. Indeed, we have
\begin{align*}
\sum_{i\in I}\sup_{\delta>0}H_\delta(u;\Q_i)\le \sum_{i\in I}\beta\left(\lvert\Q_i\rvert+\G(u;\Q_i)\right)=\beta\left(\vert O\vert+\G(u;O)\right)\sinf\infty.
\end{align*}
\end{remark}

The following result is needed for the proof of Lemma~\ref{localdirichlet}.
\begin{lemma}\label{mast:measure}{\rm(\cite{bouchitte-fonseca-mascarenhas98} and~\cite[Prop. 2.1.]{bellieud-bouchitte00})} Let $u\in W^{1,p}(\Omega;\RR^m)$. If there exists a finite Radon measure $\mu_u$ on $\Omega$ such that for every cube $\Q\in\O(\Omega)$
\begin{align*}
\Mm(u;\Q)\le \mu_u(\Q)
\end{align*}
then $\Mm^\ast(u;\cdot)$ can be extended to a Radon measure $\lambda_u$ on $\Omega$ satisfying $0\le\lambda_u\le \mu_u$.
\end{lemma} 

\section{Proof of Proposition~\ref{lsc-of-HW}}\label{proof-lsc-hw} The function $\H\W$ is ru-usc since Proposition~\ref{Coro-Prop-ru-usc2}, and so $\widehat{\H\W}$ is ru-usc by Lemma~\ref{Extension-Result-for-ru-usc-Functions}~\ref{Extension-Result-for-ru-usc-Functions2}. Since Remark~\ref{domhw} we have $\dom\H\W=\GG$. It is easy to deduce that $\dom\widehat{\H\W}=\overline\GG$. From Lemma~\ref{Extension-Result-for-ru-usc-Functions}~\ref{Extension-Result-for-ru-usc-Functions1} it holds that 
\begin{align*}
\widehat{\H\W}(\xi)=\lim_{t\to 1^-}\H\W(t\xi)\mbox{ for all }\xi\in\overline{\GG}.
\end{align*}
The proof is complete.
\hfill$\blacksquare$


\section{Proof of Proposition~\ref{main result}~\ref{mr1}}\label{proof-mainresult1}
Let $O\in\O(\Omega)$ and let $u\in W^{1,p}(\Omega;\RR^m)$ be such that $I_-(u;O)\sinf\infty$. It follows that $\nabla u(\cdot)\in\overline{\GG}$ a.e. in $O$ since $\GG$ is convex and the coercivity condition \ref{A1}. 

We have to prove that 
\begin{equation}\label{GHT-eq1}
I_-(u;O)\geq \int_O\widehat{\H\W}(\nabla u(x))dx.
\end{equation}
Consider $\{u_\eps\}_{\eps>0}\subset W^{1,p}(\Omega;\RR^m)$ satisfying $\|u_\eps-u\|_{L^p(\Omega;\RR^m)}\to 0$. 
Without loss of generality we can assume that 
\begin{equation}\label{lower-eq}
\liminf_{\eps\to 0}I_\eps(u_\eps;O)=\lim_{\eps\to 0}I_\eps(u_\eps;O)\sinf\infty, \hbox{ and so }\sup_{\eps>0}I_\eps(u_\eps;O)\sinf\infty. 
\end{equation}
Then for every $t\in[0,1[$ it holds
\begin{equation}\label{Def-I-epseq1}
t\nabla u_\eps(x)\in\GG\mbox{ for all }\eps\ssup0\hbox{ and for a.a. }x\in O
\end{equation}
and, up to a subsequence,
\begin{equation}\label{weak-W1p-subsequence}
u_\eps\wto u\hbox{ in }W^{1,p}(\Omega;\RR^m)
\end{equation}
since $\W$ is $p$-coercive and \eqref{lower-eq}. 

As $p\ssup d$, by the Sobolev compact imbedding and \eqref{weak-W1p-subsequence}, we have, for a subsequence, that
\begin{equation}\label{Embedding-EquA}
\|u_\eps-u\|_{L^\infty(\Omega;\RR^m)}\to 0.
\end{equation}
\subsection*{Step 1: Localization} For each $\eps>0$, we define the nonnegative Radon measure $\mu_\eps$ on $O$ by 
\[
\mu_\eps:=W\left({\cdot\over\eps},\nabla u_\eps(\cdot)\right)dx\lfloor_O.
\]
From (\ref{lower-eq}) we see that $\sup_\eps\mu_\eps(O)<\infty$, and so  there exists a Radon measure $\mu$ on $O$ such that (up to a subsequence) $\mu_\eps\stackrel{\ast}{\wto}\mu$. By Lebesgue's decomposition theorem, we have $\mu=\mu_a+\mu_s$ where $\mu_a$ and $\mu_s$ are nonnegative Radon measures such that $\mu_a\ll dx\lfloor_O$ and $\mu_s\perp dx\lfloor_O$, and from Radon-Nikodym's theorem we deduce that there exists $f\in L^1(O;[0,\infty[)$, given by
\begin{equation}\label{RaDoN-NiKOdYm-ForMULA-HomogeniZAtION-I}
f(x)=\lim_{\rho\to 0}{\mu_a(Q_\rho(x))\over\rho^d}=\lim_{\rho\to 0}{\mu(Q_\rho(x))\over\rho^d}\quad\hbox{ a.e. in }O 
\end{equation}
with $Q_\rho(x):=x+\rho Y$, such that
\[
\mu_a(A)=\int_A f(x)dx\hbox{ for all measurable sets }A\subset O.
\]
\begin{remark}\label{SptMu-s-Remark}
The support of $\mu_s$, ${\spt}(\mu_s)$, is the smallest closed subset $F$ of $O$ such that $\mu_s(O\ssetminus F)=0$. Hence, $O\ssetminus\spt(\mu_s)$ is an open set, and so, given any $x\in O\ssetminus\spt(\mu_s)$, there exists $\hat{\rho}>0$ such that $\overline{Q}_{\hat{\rho}}(x)\subset O\ssetminus\spt(\mu_s)$ with $\overline{Q}_{\hat{\rho}}(x):=x+\hat{\rho}\overline{Y}$. Thus, for a.a. $x\in\Omega$, $\mu(Q_\rho(x))=\mu_a(Q_{\rho}(x))$ for all $\rho>0$ sufficiently small.
\end{remark}
To prove (\ref{GHT-eq1}) it suffices to show that
\begin{equation}\label{LoWeRBoUND-HomoGENIzaTiON-EqUa2}
f(x)\ge{\widehat{\H\W}}(\nabla u(x))\quad\hbox{ a.e. in }O.
\end{equation}
Indeed, by Alexandrov theorem (see Appendix~\ref{alexandrov}) we see that 
\[
\liminf_{\eps\to 0}I_\eps(u_\eps;O)=\liminf_{\eps\to 0}\mu_\eps(O)\geq\mu(O)=\mu_a(O)+\mu_s(O)\geq\mu_a(O)=\int_\Omega f(x)dx.
\]
But, by (\ref{LoWeRBoUND-HomoGENIzaTiON-EqUa2}), we have 
\[
\int_O f(x)dx\geq\int_O{\widehat{\H\W}}(\nabla u(x))dx,
\]
and (\ref{GHT-eq1}) follows.

Fix $t\in ]0,1[$. Let $\t\in ]t,1[$ be such that
\begin{align}\label{deltafini}
\Delta^{a}_\W(\t)\sinf\infty.
\end{align} 

Fix $x_0\in O\ssetminus N$, where $N\subset O$ is a suitable set such that $|N|=0$, and such that
\begin{flalign}
&\nabla u(x_0)\in \overline{\GG};\label{domain-nabla}&\\
&f(x_0)\sinf\infty;\label{fx0}&\\
&G(\t\nabla u(x_0))\sinf\infty;\label{Gt2}&\\
&\lim_{\rho\to 0}{\frac{1}{\rho}}\Vert u-u_{x_0}\Vert_{L^\infty(\Q_{\rho}(x_0);\RR^m)}=0\label{difflowerbound}.&
\end{flalign}

Note that $G(\t\nabla u(\cdot))\sinf\infty$ a.e. in $O$ since Remark~\ref{convexity-eff-domain-remark}~\ref{convexity-eff-domain} and $\nabla u(\cdot)\in\overline{\GG}$ a.e. in $O$. Note also that $u$ is almost everywhere differentiable, i.e., $\lim_{\rho\to 0}{\frac{1}{\rho}}\Vert u-u_{x}\Vert_{L^\infty(\Q_{\rho}(x);\RR^m)}=0$ a.e. in $O$ since $p\ssup d$ (where $u_x(\cdot):=u(x)+\nabla u(x)(\cdot-x)$ is the affine tangent map of $u$ at $x\in O$).

\medskip

We have to prove that $f(x_0)\ge{\widehat{\H\W}}(\nabla u(x_0))$. 

\medskip

As $\mu(O)\sinf\infty$ we have $\mu(\partial \Q_{\rho}(x_0))=0$ for all $\rho\in]0,1]\ssetminus D$ where $D$ is a countable set. From (\ref{RaDoN-NiKOdYm-ForMULA-HomogeniZAtION-I}) and Alexandrov theorem (see Appendix~\ref{alexandrov}) we deduce that
\[
f(x_0)=\lim_{\rho\to 0}{\mu(\Q_\rho(x_0))\over\rho^d}=\lim_{\rho\to 0}\lim_{\eps\to0}{\mu_\eps(\Q_\rho(x_0))\over\rho^d},
\]
and so we are reduced to show that
\begin{equation}\label{GHT-eq1-bis}
\lim_{\rho\to 0}\lim_{\eps\to0}\fint_{\Q_{\rho}(x_0)}W\left({x\over\eps},\nabla u_\eps(x)\right)dx\geq{\widehat{\H\W}}(\nabla u(x_0)).
\end{equation}
Using ru-usc property of $\W$ we can see that
\begin{align*}
\lim_{\rho\to 0}\lim_{\eps\to0}\fint_{\Q_{\rho}(x_0)}W\left({x\over\eps},\nabla u_\eps(x)\right)dx\ge \limsup_{t\to 1^-}\limsup_{\rho\to 0}\limsup_{\eps\to0}\fint_{\Q_{\rho}(x_0)}W\left({x\over\eps},t\nabla u_\eps(x)\right)dx.
\end{align*}
So to prove \eqref{GHT-eq1-bis}, it is enough to show that
\begin{equation}\label{GHT-eq1-bisbisbis}
\limsup_{t\to 1^-}\limsup_{\rho\to 0}\limsup_{\eps\to0}\fint_{\Q_{\rho}(x_0)}W\left({x\over\eps},t\nabla u_\eps(x)\right)dx\ge{\widehat{\H\W}}(\nabla u(x_0)).
\end{equation}

\subsection*{Step 2: Cut-off method to substitute $tu_\eps$ with~$tv_\eps\!\!\in tu_{x_0}\!+\!W^{1,p}_0\!(\Q_{\rho}(\!x_0\!);\!\RR^m)$} Fix any $\eps\ssup0$ and any $s\in]0,1[$. Let $\phi\in W^{1,\infty}_0(\Q_{\rho}(x_0);[0,1])$ be a cut-off function between ${\Q_{s\rho}(x_0)}$ and $\Q_{\rho}(x_0)$ such that $\|\nabla\phi\|_{L^\infty(\Q{\rho}(x_0))}\leq{4\over \rho(1-s)}$. Setting
\[
v_\eps:=\phi u_\eps+(1-\phi)u_{x_0}
\]                        
where $u_{x_0}(\cdot):=u(x_0)+\nabla u(x_0)(\cdot-x_0)$, we have 
\begin{align}\label{FoR-SubaDDiTivity-Argument}
tv_\eps\in tu_{x_0}+W^{1,p}_0(\Q_{\rho}(x_0);\RR^m),
\end{align}
and if $\tau:={t\over\t}\in ]0,1[$ then
\begin{equation}\label{cut-offEq1}
t\nabla v_\eps:=\left\{
\begin{array}{ll}
t\nabla u_\eps&\hbox{on }\Q_{s\rho}(x_0)\\
\tau\left(\phi \t\nabla u_\eps+(1-\phi)\t\nabla u(x_0)\right)+(1-\tau)\Phi_{\eps,\rho}&\hbox{on }S_\rho
\end{array}
\right.
\end{equation}
with $S_{\rho}:=\Q_{\rho}(x_0)\ssetminus \overline{\Q_{s\rho}(x_0)}$ and $\Phi_{\eps,\rho}:={t\over 1-\tau}\nabla\phi\otimes\left(u_\eps-u_{x_0}\right)$. 

\medskip

Using the $G$-growth conditions~\ref{B1} we have
\begin{align*}
&\fint_{\Q_{\rho}(x_0)} W\left({x\over\eps},t\nabla v_\eps\right)dx\\
\leq&\fint_{\Q_{s\rho}(x_0)} W\left({x\over\eps},t\nabla u_\eps\right)dx+{1\over\rho^d}\int_{S_\rho}W\left({x\over\eps},t\nabla v_\eps\right)dx\\
\leq&\fint_{\Q_{\rho}(x_0)} W\left({x\over\eps},t\nabla u_\eps\right)dx+\beta(1-s^d)+{\beta\over\rho^d}\int_{S_\rho}G(t\nabla v_\eps)dx.
\end{align*}
On the other hand, taking \eqref{cut-offEq1} into account and using \ref{A3}, we have
\begin{align*}
G(t\nabla v_\eps)\le& 2C_1\left(1+G(\t\nabla u_\eps)+G(\t\nabla u(x_0))+G(\Phi_{\eps,\rho})\right)\\
\le& 2C_1\left(1+{1\over\alpha}W\left({x\over\eps},\t\nabla u_\eps\right)+G(\t\nabla u(x_0))+G\left(\Phi_{\eps,\rho}\right)\right)
\end{align*}
with $C_1:=C^2+C$. Moreover, it is easy to see that
\begin{align*}
&\left\|\Phi_{\eps,\rho}\right\|_{L^\infty(\Q_{\rho}(x_0);\MM^{m\times d})}\\
\leq&{4t\over(1-\tau)(1-s)}{1\over\rho}\|u-u_{x_0}\|_{L^\infty(\Q_{\rho}(x_0);\RR^m)}+{4t\over\rho(1-\tau)(1-s)}\|u_\eps-u\|_{L^\infty(\Omega;\RR^m)}
\end{align*}
where
\begin{equation}\label{G-conVeX-lim1}
\lim_{\rho\to0}{4t\over(1-t)(1-s)}{1\over\rho}\|u-u_{x_0}\|_{L^\infty(\Q_{\rho}(x_0);\RR^m)}=0
\end{equation}
since \eqref{difflowerbound}, i.e., $\lim_{\rho\to 0}{1\over\rho}\|u-u_{x_0}\|_{L^\infty(\Q_{\rho}(x_0);\RR^m)}=0$, and 
\begin{equation}\label{G-conVeX-lim2}
\lim_{\eps\to0}{4t\over\rho(1-t)(1-s)}\|u_\eps-u\|_{L^\infty(\Omega;\RR^m)}=0\hbox{ for all }\rho>0
\end{equation}
since \eqref{Embedding-EquA}, i.e., $\lim_{\eps\to0}\|u_\eps-u\|_{L^\infty(\Omega;\RR^m)}=0$. 
By Lemma~\ref{lemma finite integrand} we have for some $\rho_0\ssup 0$
\[
\maxw:=\sup_{\xi\in \overline{B}_{\rho_0}(0)}G(\xi)\sinf\infty.
\] 
By \eqref{G-conVeX-lim1} there exists $\overline{\rho}>0$ such that ${4t\over(1-t)(1-s)}{1\over\rho}\|u-u_{x_0}\|_{L^\infty(Q_{\bar\rho}(x_0);\RR^m)}<{\rho_0\over 2}$ for all $\rho\in ]0,\overline{\rho}[$.

Fix any $\rho\in ]0,\overline{\rho}[$. Taking \eqref{G-conVeX-lim2} into account we can assert that there exists $\eps_\rho>0$ such that for every $\eps\in ]0,\eps_\rho[$
\[
G\left(\Phi_{\eps,\rho}\right)\leq \maxw\hbox{ a.e. in }\Q_\rho(x_0).
\] 
Thus, for every $\eps\in ]0,\eps_\rho[$, we have 
\begin{align}\label{End-Step2-Equ}
&\fint_{\Q_{\rho}(x_0)} \W\left({x\over\eps},t\nabla v_\eps\right)dx\\
\le&\fint_{\Q_{\rho}(x_0)} \W\left({x\over\eps},t\nabla u_\eps\right)dx+\left(1-s^d\right)\Big(\beta+\overline{C}\big(1+\maxw+G(\t\nabla u(x_0))\big)\Big)\notag\\
&\hspace{3.75cm}+\frac{\overline{C}}{\alpha}\frac{1}{\rho^d}\int_{S_\rho}\W\left({x\over\eps},\t\nabla u_\eps\right)dx\notag
\end{align}
where $2\beta C_1:=\overline{C}$.
Since $\W$ is periodically ru-usc, for every $\eps\in ]0,\eps_\rho[$ we have the estimate for the last term of \eqref{End-Step2-Equ} shown as follows
\begin{align}\label{End-Step2-Equu}
\frac{1}{\rho^d}\int_{S_\rho}\W\left({x\over\eps},\t\nabla u_\eps\right)dx\le \Delta_{\W}^a(\t)\frac{1}{\rho^d}\int_{S_\rho} a\left(\frac{x}{\eps}\right)dx+\Big(1+\Delta_{\W}^a(\t)\Big)\frac{1}{\rho^d}\mu_\eps(S_\rho).
\end{align}


\subsection*{Step 3: End of the proof} Taking \eqref{FoR-SubaDDiTivity-Argument} into account we see that for every $\eps\in ]0,\eps_\rho[$
\[
\fint_{\Q_{\rho}(x_0)} \W\left({x\over\eps},t\nabla v_\eps\right)dx\geq{1\over\lvert{1\over\eps}\Q_{\rho}(x_0)\rvert} \mathcal{S}_{t\nabla u(x_0)}\left({1\over\eps}\Q_{\rho}(x_0)\right),
\]
where $\mathcal{S}_\xi(A)$ is given by \eqref{SuBaDDiTiVe-W} for all $\xi\in\MM^{m\times d}$ and all open set $A\subset\RR^d$. By \eqref{domain-nabla} we have $\nabla u(x_0)\in\overline{\GG}$, and so $t\nabla u(x_0)\in\GG$ because $\GG$ is convex and $0\in{\rm int}(\GG)$ since \ref{A2} and \ref{A3}. From Corollary~\ref{SubadditiveTheorem} we deduce that
\begin{equation}\label{PassingToTheLimit-1}
\limsup_{\eps\to 0}\fint_{\Q_{\rho}(x_0)} W\left({x\over\eps},t\nabla v_\eps\right)dx\geq\mathcal{H}W(t   \nabla u(x_0))
\end{equation}
for all $\rho\in]0,\overline{\rho}[$. On the other hand, as $\mu_\eps(S_\rho)\leq\mu_\eps(\overline{S}_\rho)$ for all $\eps\in ]0,\eps_\rho[$, $\overline{S}_\rho$ is compact and $\mu_\eps\mwto\mu$, we have $\limsup_{\eps\to0}\mu_\eps(S_\rho)\leq\mu(\overline{S}_\rho)$ by Alexandrov theorem. But $\mu(\overline{S}_\rho)=\mu_a(\overline{S}_\rho)$ since $\overline{S}_\rho\subset\overline{Q}_\rho(x_0)\subset\Omega\ssetminus\spt(\mu_s)$ (see Remark~\ref{SptMu-s-Remark}). Hence, for every $\rho\in ]0,\overline{\rho}[$,
\[
\limsup_{\eps\to0}{1\over\rho^d}\mu_\eps(S_\rho)\leq{1\over\rho^d}\int_{S_\rho}f(x)dx=\fint_{\Q_{\rho}(x_0)}f(x)dx-s^d
\fint_{Q_{s\rho}(x_0)}f(x)dx,
\]
and consequently
\begin{equation}\label{PassingToTheLimit-2}
\limsup_{\rho\to0}\limsup_{\eps\to0}{1\over\rho^d}\mu_\eps(S_\rho)\le(1-s^d)f(x_0).
\end{equation}
Taking \eqref{End-Step2-Equ} and  \eqref{End-Step2-Equu} into account, from \eqref{PassingToTheLimit-1} and \eqref{PassingToTheLimit-2} we deduce that
\begin{align*}
&\H\W(t\nabla u(x_0))\\
\le& \limsup_{\eps\to0}\fint_{\Q_{\rho}(x_0)} W\left({x\over\eps},t\nabla u_\eps\right)dx\\+&\left(1-s^d\right)\left(\beta+\overline{C}\left(1+\maxw+G(\t\nabla u(x_0)+\frac1\alpha\left(\Delta_\W^a(\t)\langle a\rangle+f(x_0)\right)\right)\right).
\end{align*}
Taking \eqref{fx0}, \eqref{Gt2} and \eqref{deltafini} into account and passing to the limits $\rho\to 0$ and $s\to 1$, we obtain
\[
\H\W(t\nabla u(x_0))\le\limsup_{\rho\to0}\limsup_{\eps\to0}\fint_{\Q_{\rho}(x_0)} \W\left({x\over\eps},t\nabla u_\eps\right)dx,
\]
and \eqref{GHT-eq1-bisbisbis} follows when $t\to 1$. \hfill$\blacksquare$


\section{Proof of Proposition~\ref{main result}\ref{mr2}}\label{proof-mainresult2}

For each $(u,O)\in W^{1,p}(\Omega;\RR^m)\times\O(\Omega)$ we recall that 
\begin{align*}
\M_\eps(u;O):=\inf\left\{I_\eps(v;O):v\in W^{1,p}_0(O;\RR^m)\right\}\quad\mbox{ and }\quad
\Mm(u;O):=\limsup_{\eps\to 0} \M_\eps(u;O).
\end{align*}

We give a sketch of the proof which is divided into three steps.

The first step consists in proving that ${I_+(u;O)\le\Mm^\ast(u;O)}$ for all $(u,O)\in W^{1,p}(\Omega;\RR^m)\times \O(\Omega)$. When we assume that $\G(u;O)\sinf\infty$, Lemma~\ref{mast:measure} and the $G$-growth conditions imply that $\Mm^\ast(u;\cdot)$ is a Radon measure which is absolutely continuous with respect to the Lebesgue measure on $O$. Thus, we can write
\begin{align}\label{eq1r}
I_+(u;O)\le \Mm^\ast(u;O)=\int_O\lim_{\rho\to 0}\frac{\Mm^\ast(u;\Q_{\rho}(x))}{\rho^d}dx.
\end{align} 

The second step consists in showing that $\Mm^\ast(u;\cdot)$ is locally equivalent to $\Mm(u;\cdot)$, i.e., for a.a. $x\in O$
\begin{align}\label{eq2r}
\lim_{\rho\to 0}\frac{\Mm^\ast(u;\Q_{\rho}(x))}{\rho^d}=\lim_{\rho\to 0}\frac{\Mm(u;\Q_{\rho}(x))}{\rho^d}.
\end{align}
This is carrying out by measure theoretic arguments (see Step 2).

In the third and last step we replace $u$ by $tu$ with $t\in]0,1[$ and we show, using cut-off techniques, that for a.a. $x\in O$
\begin{align}\label{eq3r}
\lim_{\rho\to 0}\frac{\Mm(tu;\Q_{\rho}(x))}{\rho^d}\le\liminf_{s\to 1}\limsup_{\rho\to 0}\frac{\Mm(tu_x;\Q_{s\rho}(x))}{(s\rho)^d}
\end{align}
where $u_x(\cdot):=u(x)+\nabla u(x)(\cdot-x)$. The right hand term of \eqref{eq3r} is equal to $\H\W(t\nabla u(x))$. Indeed, it is easy to see that for any $\eps,\rho\ssup 0$ and any $x\in O$ we can write 
\begin{align*}
&\frac{\M_\eps(tu_x;\Q_{s\rho}(x))}{(s\rho)^d}\\
=&\frac{1}{\left\vert\frac{1}{\eps}\Q_{s\rho}(x)\right\vert}\inf\left\{\int_{\frac{1}{\eps}\Q_{s\rho}(x)} \W(y,\nabla v(\eps y))dy:v\in tu_x+W^{1,p}_0(\Q_{s\rho}(x);\RR^m)\right\}\\
=&\frac{1}{\left\vert\frac{1}{\eps}\Q_{s\rho}(x)\right\vert}\inf\left\{\int_{\frac{1}{\eps}\Q_{s\rho}(x)} \W(y,t\nabla u(x_0)+\nabla \phi(\eps y))dy:\phi\in W^{1,p}_0(\Q_{s\rho}(x);\RR^m)\right\}\\
=&\frac{1}{\left\vert\frac{1}{\eps}\Q_{s\rho}(x)\right\vert}\S_{t\nabla u(x)}\left(\frac{1}{\eps}\Q_{s\rho}(x)\right)
\end{align*}
which give
\begin{align*}
\frac{\Mm(tu_x;\Q_{s\rho}(x))}{(s\rho)^d}=\H\W(t\nabla u(x))
\end{align*}
since a subadditive argument (see Corollary~\ref{SubadditiveTheorem}). The proof is achieved by taking \eqref{eq2r} and \eqref{eq3r} into account in the inequality \eqref{eq1r}
\begin{align*}
I_+(tu;O)\le \int_O\lim_{\rho\to 0}\frac{\Mm^\ast(tu;\Q_{\rho}(x))}{\rho^d}dx&=\int_O \lim_{\rho\to 0}\frac{\Mm(tu;\Q_{\rho}(x))}{\rho^d}dx\\
&\le \int_O\liminf_{s\to 1}\limsup_{\rho\to 0}\frac{\Mm(tu_x;\Q_{s\rho}(x))}{(s\rho)^d}dx\\
&=\int_O \H\W(t\nabla u(x))dx.
\end{align*}
\bigskip

\subsection*{Step 1: Prove that ${I_+(u;O)\le\Mm^\ast(u;O)}$ when $\G(u;O)\sinf\infty$}\label{subsect1} 
Fix $(u,O)\in W^{1,p}(\Omega;\RR^m)\times\O(\Omega)$ such that $\G(u;O)\sinf\infty$. Without loss of generality we assume that 
\begin{align}\label{eq0: step1}
\Mm^\ast(u;O)\sinf\infty.
\end{align}

Fix $\eps\in ]0,1[$. Choose $\{\overline{\Q}_i\}_{i\in I}\in\mathcal{V}_\eps(O)$ such that
\begin{align}\label{eq1: step1}
\sum_{i\in I}\Mm(u;\Q_i)\le \Mm^\eps(u;O)+\frac\eps2\le \Mm^\ast(u;O)+\frac\eps2.
\end{align}
Fix $\delta\in ]0,1[$. Given any $i\in I$ there exists $v_i\in u+{W}^{1,p}_0(\Q_i;\RR^m)$ such that 
\begin{align}\label{eq2: step1}
I_\delta(v_i;\Q_i)\le \M_\delta(u;\Q_i)+\frac\delta2 \frac{\vert \Q_i\vert}{\vert O\vert}
\end{align}
by definition of $\M_\delta(u;\Q_i)$.
Define $u_{\delta,\eps}\in u+W^{1,p}_0(O;\RR^m)$ by 
\[
u_{\delta,\eps}:=\sum_{i\in I} v_i\mathbb{I}_{\Q_i}+u\mathbb{I}_{\Omega\setminus\mathop{\cup}_{i\in I}\Q_i}.
\] 
From \eqref{eq2: step1} we have that
\begin{align*}
I_\delta(u_{\delta,\eps};O)=\sum_{i\in I}I_\delta(v_i;\Q_i)\le \sum_{i\in I}\M_\delta(u;\Q_i)+\frac\delta2.
\end{align*}
Letting $\delta\to 0$ we obtain
\begin{align}\label{eq5: step1}
\limsup_{\delta\to 0}I_\eps(u_{\delta,\eps};O)\le \limsup_{\delta\to 0}\sum_{i\in I}\M_\delta(u;\Q_i).
\end{align}
By the $G$-growth conditions~\ref{B1} we have $\sup\limits_{\eta\in ]0,\delta[}\M_\eta(u;\Q_i)\le\beta(\vert \Q_i\vert+\G(u;\Q_i))$ for all $\delta\ssup0$ and all $i\in I$ with 
\begin{align*}
\sum_{i\in I}\beta(\vert \Q_i\vert+\G(u;\Q_i))=\beta(\vert O\vert+\G(u;O))\sinf\infty,
\end{align*}
then applying the dominated convergence theorem we have
\begin{align}\label{eq6: step1}
\limsup_{\delta\to 0}\sum_{i\in I}\M_\delta(u;\Q_i)\le\sum_{i\in I}\Mm(u;\Q_i).
\end{align}
Therefore collecting \eqref{eq1: step1}, \eqref{eq5: step1}, \eqref{eq6: step1} and passing to the limit $\eps\to 0$, we have
\begin{align}\label{eq7: step1}
\limsup_{\eps\to 0}\limsup_{\delta\to 0}I_\delta(u_{\delta,\eps};O)\le \Mm^\ast(u;O).
\end{align}
From the $p$-coercivity of $\W$ \ref{A1}, \eqref{eq7: step1} and \eqref{eq0: step1}, we deduce
\begin{align}\label{eq3: step1}
\limsup_{\eps\to 0}\limsup_{\delta\to 0}\int_O \vert\nabla u_{\delta,\eps}\vert^pdx\sinf\infty.
\end{align}
By Poincar\'e inequality there exists $K\ssup 0$ depending only on $p$ and $d$ such that for each $v_i\in u+{W}^{1,p}_0(\Q_i;\RR^m)$
\begin{align*}
\int_{\Q_i}\vert v_i-u\vert^p dx\le K\eps^p\int_{\Q_i}\vert \nabla v_i-\nabla u\vert^pdx
\end{align*}
since $\diam(\Q_i)\sinf\eps$. Summing on $i\in I$ 
we obtain
\begin{align*}
\int_O\vert u_{\delta,\eps}-u\vert^p dx\le 2^{p-1}K\eps^p\left(\int_{O}\vert \nabla u_{\delta,\eps}\vert^p dx+\int_O \vert\nabla u\vert^pdx\right)
\end{align*}
which shows that 
\begin{align}\label{eq9: step1}
\limsup_{\eps\to 0}\limsup_{\delta\to 0}\int_\Omega\vert u_{\delta,\eps}-u\vert^p dx=0
\end{align}
since~\eqref{eq3: step1}.
A simultaneous diagonalization of \eqref{eq7: step1} and \eqref{eq9: step1} gives a sequence $\{u_\delta:=u_{\delta,\eps(\delta)}\}_{\delta>0}\subset u+W^{1,p}_0(O;\RR^m)$ such that $u_\delta\to u$ in $L^p(\Omega;\RR^m)$ and 
\begin{align*}
I_+(u;O)\le \limsup_{\delta\to 0}I_{\delta}(u_{\delta};O)\le \Mm^\ast(u;O)
\end{align*}
since the definition of $I_+(u;O)$. The proof is complete.

\hfill$\blacksquare$


\subsection*{Step 2: Prove that $\displaystyle\Mm^\ast(u;\cdot)$ is locally equivalent to $\Mm(u;\cdot)$}\label{subsect2}
In this step we use the following result from~\cite{bouchitte-fonseca-mascarenhas98, bellieud-bouchitte00}. For a sake of completeness we give a proof.
\begin{lemma}\label{localdirichlet} If $\G(u;O)\sinf\infty$ then we have 
\[
\displaystyle\lim_{\rho\to 0}\frac{\Mm^\ast(u;\Q_\rho(x_0))}{\rho^d}=\lim_{\rho\to 0}\frac{\Mm(u;\Q_\rho(x_0))}{\rho^d} \quad x_0\mbox{-a.e. in }O.
\]
\end{lemma}
\begin{proof}Let $u\in W^{1,p}(\Omega;\RR^m)$ be such that $\G(u;O)\sinf\infty$. Then for each $U\in\O(O)$
\begin{align*}
\Mm(u;U)\le \limsup_{\eps\to 0}\int_U \W(\frac{x}{\eps},\nabla u(x))dx\le\beta\left(\vert O\vert+\int_O G(\nabla u)dx\right)\sinf\infty,
\end{align*}
so, using Lemma~\ref{mast:measure} with $\mu_u:=\beta\left(\vert\cdot\vert+G(\nabla u(\cdot))dx)\right)\lfloor_{O}$, we have $\Mm(u;\cdot)$ is the trace of a Radon measure $\lambda_u$ on $O$ satisfying $0\le\lambda_u\le \mu_u$. Since $\mu_u$ is absolutely continuous with respect to $dx\lfloor_{O}$ the Lebesgue measure on $O$, the limit $\lim_{\rho\to 0}\frac{\lambda_u(\Q_\rho(x_0))}{\rho^d}$ exists for a.a. $x_0\in O$ as the Radon-Nikodym derivative of $\lambda_u$ with respect to $dx\lfloor_{O}$. Moreover, using Lemma~\ref{mastinquelity}, the $G$-growth conditions together with Remark~\ref{condition-domination-mstar} we have
\begin{align*}
\lim_{\rho\to 0}\frac{\Mm^\ast(u;\Q_\rho(x_0))}{\rho^d}\ge \limsup_{\rho\to 0}\frac{\Mm(u;\Q_\rho(x_0))}{\rho^d}\; x_0\mbox{-a.e. in }O.
\end{align*}
It remains to prove that 
\begin{align}\label{eq-1:localsup}
\lim_{\rho\to 0}\frac{\Mm^\ast(u;\Q_\rho(x_0))}{\rho^d}\le \liminf_{\rho\to 0}\frac{\Mm(u;\Q_\rho(x_0))}{\rho^d}\; x_0\mbox{-a.e. in }O.
\end{align}
Fix any $\theta\ssup 0$. Consider the following sets
\begin{align*}
\mathcal{G}_\theta:=&\Big\{ \Q_\rho(x): x\in O,\;\rho >0\;\mbox{ and }\;\Mm^\ast(u;\Q_\rho(x))\ssup \Mm(u;\Q_\rho(x))+\theta\left\vert\Q_\rho(x)\right\vert\Big\},\\
\mathcal{N}_\theta:=&\Big\{x\in O:\forall \delta\ssup0\;\;\exists\rho \in ]0,\delta[\;\;\Q_\rho(x)\in\mathcal{G}_\theta\Big\}.
\end{align*}
It is sufficient to prove that $\mathcal{N}_\theta$ is a negligible set for the Lebesgue measure on $ O$. Indeed, given $x_0\in O\ssetminus\mathcal{N}_\theta$ there exists $\delta_0\ssup 0$ such that \[\Mm^\ast(u;\Q_{\rho}(x_0))\le \Mm(u;\Q_{\rho}(x_0))+\theta\left\vert\Q_{\rho}(x_0)\right\vert\] for all $\rho\in]0,\delta_0[$. Hence 
\begin{align*}
\displaystyle\lim_{\rho\to 0}\frac{\Mm^\ast(u;\Q_\rho(x_0))}{\left\vert\Q_\rho(x_0)\right\vert}\le\liminf_{\rho\to 0}\frac{\Mm(u;\Q_{\rho}(x_0))}{\left\vert\Q_{\rho}(x_0)\right\vert}+\theta,
\end{align*}
then we obtain \eqref{eq-1:localsup} when $\theta\to 0$.

Fix $\delta\ssup 0$. Consider the set 
\begin{align*}
\mathcal{F}_\delta:=\Big\{\overline{\Q}_\rho(x): x\in\mathcal{N}_\theta,\;\rho \in]0,\delta[\;\mbox{ and }\Q_\rho(x)\in\mathcal{G}_\theta\Big\}.
\end{align*}
Using the definition of $\mathcal{N}_\theta$ we can see that $\inf_{\Q\in\mathcal{F}_\delta}\diam\left({\Q}\right)=0$. By the Vitali covering theorem there exists a (pairwise) disjointed countable subfamily $\{\overline{\Q}_i\}_{i\ge 1}$ of $\mathcal{F}_\delta$ such that
\begin{align}\label{eq1: localsup}
\big\vert\mathcal{N}_\theta\ssetminus \mathop{\cup}_{i\ge 1}\Q_i\big\vert=0.
\end{align}
We have $\mathcal{N}_\theta\subset \mathop{\cup}_{i\ge 1}\Q_i\cup \mathcal{N}_\theta\ssetminus \mathop{\cup}_{i\ge 1}\Q_i$. To prove that $\mathcal{N}_\theta$ is a negligible set is equivalent to prove that $\vert V_j\vert=0$ for all $j\ge 1$ where 
\[
V_j:=\mathop{\cup}_{i=1}^j\Q_i.
\]
Fix $j\ge 1$. Let $\{\Q^\prime_i\}_{i\ge 1}\in\mathcal{V}_\delta\big( O\setminus \mathop{\cup}_{i=1}^j\overline{\Q}_i\big)$ satisfying
\begin{align}\label{eq2: localsup}
\sum_{i\ge 1}\Mm(u;\Q^\prime_i)\le\Mm^\ast\big(u; O\ssetminus \mathop{\cup}_{i=1}^j\overline{\Q}_i\big)+\delta.
\end{align}
Recalling that $\Mm^\ast(u;\cdot)$ is the trace on $\O( O)$ of a nonnegative finite Radon measure, we see that 
\begin{align*}
\Mm^\ast(u; O)\ge& \Mm^\ast\big(u; O\ssetminus \mathop{\cup}_{i=1}^{j}\overline{\Q}_i\big)+\Mm^\ast\big(u;V_j\big)\\
=&\Mm^\ast\big(u; O\ssetminus \mathop{\cup}_{i=1}^{j}\overline{\Q}_i\big)+\sum_{1\le i\le j}\Mm^\ast(u;{\Q_i}).
\end{align*}
Since each $\Q_i\in\mathcal{G}_\theta$, by \eqref{eq2: localsup} we have
\begin{align*}
\Mm^\ast(u; O)\ge\sum_{i\ge 1}\Mm(u;\Q^\prime_i)-\delta+\sum_{i=1}^{j}\Mm(u;{\Q_i})+\theta\vert V_j\vert.
\end{align*}
It is easy to see that the countable family $\{\Q^\prime_i:i\ge 1\}\cup\{\Q_i:1\le i\le j\}$ belongs to $\mathcal{V}_\delta( O)$, thus
\begin{align*}
\Mm^\ast(u; O)\ge \Mm^\delta(u; O)+\theta\vert V_j\vert-\delta.
\end{align*}
Letting $\delta\to 0$, we have $\Mm^\delta(u; O)\rightarrow \Mm^\ast(u; O)$, and so $\vert V_j\vert=0$ since $\theta\ssup 0$.
\end{proof}


\subsection*{Step 3: Cut-off technique to locally substitute $tu$ with $tu_{x_0}$ in $\frac{\Mm(\cdot;\Q_{\rho}(x_0))}{\rho^d}$}\label{subsect3}
Fix $(u,O)\in W^{1,p}(\Omega;\RR^m)\times\O(\Omega)$. Let $t\in ]0,1[$ such that  
\begin{align}\label{step3-eq0}
\G(tu;O)\sinf\infty.
\end{align}
Our goal here is to prove 
\begin{align}\label{upperboundlocal}
\lim_{\rho\to 0}\frac{\Mm(tu;\Q_\rho(x_0))}{\rho^d}\le \H\W(t\nabla u(x_0))\quad x_0\hbox{-a.e. in }O.
\end{align}
We claim that it is enough to prove that
\begin{align}\label{Step3-eq1}
\lim_{\rho\to 0}\frac{\Mm(tu;\Q_\rho(x_0))}{\rho^d}\le \liminf_{s\to 1}\limsup_{\rho\to 0}\frac{\Mm(tu_{x_0};\Q_{s\rho}(x_0))}{\left(s\rho\right)^d}\quad x_0\hbox{-a.e. in }O,
\end{align}
where $u_{x_0}(\cdot):=u(x_0)+\nabla u(x_0)(\cdot-x_0)$. Indeed, by Corollary~\ref{SubadditiveTheorem} we have
\begin{align*}
&\liminf_{s\to 1}\limsup_{\rho\to 0}\frac{\Mm(tu_{x_0};\Q_{s\rho}(x_0))}{\left(s\rho\right)^d}\\
=& \liminf_{s\to 1}\limsup_{\rho\to 0}\limsup_{\eps\to 0}\frac{1}{\left\vert\frac1\eps\Q_{s\rho}(x_0)\right\vert}{\S_{t\nabla u(x_0)}\left(\frac1\eps\Q_{s\rho}(x_0)\right)}=\H\W(t\nabla u(x_0))\quad x_0\hbox{-a.e. in }O.
\end{align*} 
since $t\nabla u(\cdot)\in\GG$ a.e. in $O$ by \eqref{step3-eq0}. 

\bigskip

We are reduced to prove \eqref{Step3-eq1}. Consider $x_0\in O$ satisfying
\begin{flalign}
&\lim_{\rho\to 0}\fint_{\Q_{\rho}(x_0)}G(t\nabla u(x))dx=G(t\nabla u(x_0))\sinf\infty\label{Step3-eq12};&\\
&\lim_{\rho\to 0}\fint_{\Q_{\rho}(x_0)}G(\t\nabla u(x))dx=G(\t\nabla u(x_0))\sinf\infty\label{Step3-eq13}.&
\end{flalign}
Fix $s\in ]0,1[$, $\rho\ssup 0$ and $\eps\ssup 0$. 

Choose $v\in u_{x_0}+W^{1,p}_0(\Q_{s\rho}(x_0);\RR^m)$ satisfying
\begin{align*}
I_\eps(tv;\Q_{s\rho}(x_0))\le \M_\eps(tu_{x_0};\Q_{s\rho}(x_0))+\rho^{d+1}.
\end{align*}
Consider a cut-off function $\phi\in W^{1,\infty}_0(\Q_{\rho}(x_0);[0,1])$ between $\Q_{s\rho}(x_0)$ and $\Q_{\rho}(x_0)$ such that $\Vert \nabla \phi\Vert_{L^\infty(\Q_{\rho}(x_0) )}\le\frac{4}{(1-s)\rho}$.

Define $w:=\phi v+(1-\phi)u$ which belongs to $u+W^{1,p}_0(\Q_{\rho}(x_0);\RR^m)$. We have
\begin{align}\label{Step3 eq2}
\M_\eps(tu;\Q_{\rho}(x_0))&\le I_\eps(tv;\Q_{s\rho}(x_0))+I_\eps(tw;\Q_{\rho}(x_0)\ssetminus \Q_{s\rho}(x_0))\\
&\le \M_\eps(tu_{x_0};\Q_{s\rho}(x_0))+\rho^{d+1}+I_\eps(tw;\Q_{\rho}(x_0)\ssetminus \Q_{s\rho}(x_0))\notag.
\end{align}
Let us estimate the last term of \eqref{Step3 eq2} divided by $\rho^d$. 

Set $\tau:={t}{\t}^{-1}$, $\Phi_{\rho}:=\frac{t}{1-\tau}\nabla \phi\otimes(u_{x_0}-u)$ and $S_{\rho}:=\Q_{\rho}(x_0)\ssetminus\overline{\Q_{s\rho}(x_0)}$. Using the $G$-growth conditions, \ref{A3} and \eqref{Step3-eq13} we have
\begin{align}\label{Step3 eq4}
&\limsup_{\rho\to 0}\limsup_{\eps\to 0} \frac{I_\eps(tw;S_{\rho})}{\rho^d}\\
\le&\beta\limsup_{\rho\to 0}\left((1-s^d)+\frac{1}{\rho^d}\int_{S_{\rho}}G\big(\tau(\phi \t\nabla u(x_0)+(1-\phi)\t\nabla u)+(1-\tau)\Phi_{\rho}\big)dx\right)\notag\\
\le& C_1\limsup_{\rho\to 0}\Bigg((1-s^d)+\frac{1}{\rho^d}\int_{S_{\rho}}G(\phi \t\nabla u(x_0)+(1-\phi)\t\nabla u)dx\left.+\frac{1}{\rho^d}\int_{S_{\rho}}G(\Phi_{\rho})dx\right)\notag\\
\le& C_2\limsup_{\rho\to 0}\Bigg((1-s^d)(1+G(\t\nabla u(x_0)))+\frac{1}{\rho^d}\int_{S_{\rho}}G(\t\nabla u)dx\notag\left.+\frac{1}{\rho^d}\int_{S_{\rho}}G(\Phi_{\rho})dx\right)\notag\\
\le& 2C_2\left((1-s^d)(1+G(\t\nabla u(x_0)))+\limsup_{\rho\to 0}\frac{1}{\rho^d}\int_{S_{\rho}}G(\Phi_{\rho})dx\right)\notag.
\end{align}
where $C_1:=\beta(1+C)$ and $C_2:=C_1(1+C)$. 

Since \eqref{difflowerbound}, we choose $\overline{\rho}\ssup 0$ such that for every $\rho\in ]0,\overline{\rho}[$ it holds
\begin{align*}
\frac{4t}{(1-\tau)(1-s)}\frac{1}{\rho}\Vert u_{x_0}-u\Vert_{L^\infty(\Q_\rho(x_0);\RR^m)}\le \frac{\rho_0}{2}.
\end{align*}
It follows that 
\begin{align*}
\Vert \Phi_\rho\Vert_{L^\infty(\Q_\rho(x_0);\MM^{m\times d})}\le \rho_0.
\end{align*}
Using Lemma~\ref{lemma finite integrand} we deduce that
\begin{align}\label{Step3 eq5}
G(\Phi_\rho(\cdot))\le\maxw \;\mbox{ a.e. in }\Q_{\rho}(x_0)
\end{align}
for all $\rho\in ]0,\overline{\rho}[$.
From \eqref{Step3 eq4} and  \eqref{Step3 eq5} we obtain
\begin{align}\label{Step3 eq6}
\limsup_{\rho\to 0}\limsup_{\eps\to 0} \frac{I_\eps(tw;S_{\rho})}{\rho^d}\le2C_2(1-s^d)(1+G(\t\nabla u(x_0)+\maxw).
\end{align}
Taking \eqref{Step3 eq6} into account in the inequality \eqref{Step3 eq2} we obtain by passing to the limit $\eps\to 0$ and $\rho\to 0$
\begin{align*}
\lim_{\rho\to 0}\frac{\Mm(tu;\Q_{\rho}(x_0))}{\rho^d}\le s^d\limsup_{\rho\to 0}\frac{\Mm(tu_{x_0};\Q_{s\rho}(x_0))}{(s\rho)^d}\!+\!2C_2(1-s^d)(1\!+\!G(\t\nabla u(x_0))\!+\!\maxw).
\end{align*}
Letting $s\to 1$, we finally find \eqref{Step3-eq1}.\hfill$\blacksquare$

\section{Proof of Theorem~\ref{main result-cor}}\label{proof-theorem} 
To shorten notation we set
\begin{align*}
\G(\cdot;\Omega):=\G(\cdot), \quad\H I(\cdot;\Omega):=\H I(\cdot), \;\mbox{ and }\;\widehat{\H I}(\cdot;\Omega):=\widehat{\H I}(\cdot).
\end{align*}

Let $u\in W^{1,p}(\Omega;\RR^m)$. The proof of the lower bound is already done in Proposition~\ref{main result}~\ref{mr1}, so it remains to prove the upper bound $I_+(\cdot)\le \widehat{\H I}(\cdot)$. 

Assume without loss of generality that $ \widehat{\H I}(u)\sinf\infty$. 

First, we assume that $\nabla u(\cdot)\in\inte(\GG)$ a.e. in $\Omega$. In this case, from Proposition~\ref{lsc-of-HW}, coercivity conditions \ref{B1} and \ref{A3} we have 
\begin{align*}
\widehat{\H I}(u)=\int_\Omega \lim_{t\to 1^-}\H\W(t\nabla u(x))dx\ge \alpha\int_\Omega \liminf_{t\to1^-}G(t\nabla u(x))dx.
\end{align*}
But $G$ is $W^{1,p}$-quasiconvex, so it is continuous in $\inte(\GG)$ (see for instance~\cite{fonseca88} or~\cite{dacorogna08}). It follows that
\begin{align*}
\liminf_{t\to 1^-} G(t\nabla u(x))=G(\nabla u(x))\mbox{ a.e. in }\Omega
\end{align*}
since we assumed that $\nabla u(\cdot)\in\inte(\GG)$ a.e. in $\Omega$. Therefore
\begin{align*}
\alpha\G(u)\le \widehat{\H I}(u)\sinf\infty.
\end{align*}
So, using \ref{A3} and \ref{A2} we have $\G(tu)\le C\left(\vert\Omega\vert(1+G(0))+\G(u)\right)\sinf\infty$ for all $t\in ]0,1[$. Applying Proposition~\ref{main result}~\ref{mr2} we have for every $t\in ]0,1[$
\begin{align*}
I_+(tu)\le\H I(tu).
\end{align*}
Since \ref{B1}, we have for any $t\in ]0,1[$ that $\H\W(t\nabla u(\cdot))\le \beta(1+G(t\nabla u(\cdot))$, so applying Lebesgue dominated convergence theorem and Proposition~\ref{lsc-of-HW} we obtain
\begin{align*}
\liminf_{t\to 1^-}\H I(tu)\le \int_\Omega \lim_{t\to 1^-}\H \W(t\nabla u(x))dx=\widehat{\H I}(u).
\end{align*}
Using the lower semicontinuity of $I_+$ with respect to the strong topology of $L^p(\Omega;\RR^m)$, we then have
\begin{align*}
I_+(u)\le\liminf_{t\to 1^-}I_+(tu)\le\liminf_{t\to 1^-}\H I(tu)\le\widehat{\H I}(u).
\end{align*}

Now, we assume the general case $\nabla u(\cdot)\in\overline{\GG}$ a.e. in $\Omega$. We have for every $t\in]0,1[$ that $t\nabla u(\cdot)\in\inte(\GG)$ a.e. in $\Omega$ since $\GG$ is convex with $0\in\inte(\GG)$. We can apply the first part of the proof to get
\begin{align*}
I_+(tu)\le \widehat{\H I}(tu)
\end{align*}
for all $t\in ]0,1[$. But $\widehat{\H\W}$ is ru-usc since Proposition~\ref{lsc-of-HW}, so, for every $t\in]0,1[$
\begin{align*}
\widehat{\H I}(tu)\le \Delta_{\widehat{\H\W}}^{\langle a\rangle}(t)\left(\langle a\rangle\vert\Omega\vert+\widehat{\H I}(u)\right)+\widehat{\H I}(u).
\end{align*}
Letting $t\to 1$ and using the lower semicontinuity of $I_+$, we obtain the desired result.\hfill$\blacksquare$

\section{A two dimensional example}\label{example} We show, when $m=d=2$, how to construct an example of $\W$ with $G$-growth conditions satisfying the assumptions of Theorem~\ref{main result-cor}. 

Consider a set $\GG\subset\MM^{2\times 2}$ with the following properties:
\begin{hyp}
\item\label{exa1} $0\in\inte(\GG)$;
\item\label{exa2} $\GG$ is convex;
\item\label{exa3} $\det(I+\xi)\ssup 0$ for all $\xi\in\GG$;
\item\label{exa4} ${\rm tr}\left({\rm cof}(I+\xi)^\intercal(I+\zeta)\right)\ssup 0$ for all $\xi,\zeta\in\GG$,
\end{hyp}
where $I$ is the identity matrix and ${\rm cof}(F)$ is the matrix of cofactors of $F\in\MM^{2\times 2}$. 
\begin{remark}The set $I+\GG$ can be interpreted as {\em internal constraints} of an elastic material. However, the properties of $I+\GG$ do not fit with the requirements of the theory of internal constraints as developed by \cite{gurtin-podio73}. Indeed, due to the frame indifference principle, we should have 
\begin{align}\label{frame-indifferent-G}
{\rm SO(2)}(I+\GG)\subset I+\GG,
\end{align}
but this is not true. Assume that a such $I+\GG$ satisfying \eqref{frame-indifferent-G} exists then ${\rm SO(2)}\subset I+\GG$ since \ref{exa1}. Choose any rotation matrix $I+\xi$ with angle $\theta\in [\frac{\pi}{2},\frac{3\pi}{2}]$ and $I+\zeta=I$, i.e., $\zeta=0$, then 
\begin{align*}
{\rm tr}\left({\rm cof}(I+\xi)^\intercal(I+\zeta)\right)={\rm tr}\left(I+\xi \right)\le 0,
\end{align*}
so \ref{exa4} cannot be satisfied.
\end{remark}
Let $g:\MM^{2\times 2}\to [0,\infty]$ be defined by
\[
g(\xi):=\left\{
\begin{array}{cl}
\displaystyle h({\det(I+\xi)}) &\mbox{ if }\xi\in \GG\\ \\
\infty&\mbox{ otherwise}
\end{array}
\right.
\]
where $h:]0,\infty[\to [0,\infty]$ is a nonincreasing convex function satisfying for every $\lambda\in]0,1[$ and every $x\in ]0,\infty[$
\begin{align}
h(\lambda x)\le\frac{1}{\lambda^r}h(x)\label{h}
\end{align}
where $r\leq 1$. Note that the function $h$ can be chosen to satisfy $\lim\limits_{x\to 0}h(x)=\infty$.

\begin{proposition}\label{gpoly} We have 
\begin{enumerate}[label={(\roman{*})}]
\setcounter{enumi}{0}
\item\label{exa5} $\g$ is polyconvex where $\g(\cdot):=g(\cdot-I)$;
\item\label{exa6} for every $\xi,\zeta\in\GG$ and every $\lambda\in ]0,1[$ it holds 
\[
g(\lambda\xi+(1-\lambda)\zeta)\le g(\xi)+g(\zeta);
\]
\item\label{exa7} $g$ is ru-usc.
\end{enumerate}
\end{proposition}
\begin{proof} We have~\ref{exa5} because we can write $\g(F)=\varphi(F,\det(F))$ with $\varphi:\MM^{2\times 2}\times \RR\to [0,\infty]$ is the convex function defined by
\[
\varphi(F,s):=\left\{
\begin{array}{cl}
\displaystyle h(s) &\mbox{ if }F\in I+\GG\\ \\
\infty&\mbox{ otherwise.}
\end{array}
\right.
\]

Now, we show~\ref{exa6}. Fix $\xi,\zeta\in\GG$ and $\lambda\in ]0,1[$. Using~\ref{exa3},~\ref{exa4}, and properties of $h$, we have  
\begin{align}\label{indet}
&g(\lambda\xi+(1-\lambda)\zeta))\\
=&h({\det(\lambda (I+\xi)+(1-\lambda)(I+\zeta))})\notag\\
=&h(\lambda^2\det(I+\xi)+(1-\lambda)^2\det(I+\zeta)+\lambda(1-\lambda){\rm tr}\left({\rm cof}(I+\xi)^\intercal(I+\zeta)\right))\notag\\
\le&h(\lambda^2\det(I+\xi)+(1-\lambda)^2\det(I+\zeta))\notag\\
\le&\lambda h(\lambda\det(I+\xi))+(1-\lambda)h((1-\lambda)\det(I+\zeta))\notag\\
\le&\lambda^{1-r} h(\det(I+\xi))+(1-\lambda)^{1-r}h(\det(I+\zeta))\notag\\
\le&g(\xi)+g(\zeta)\notag.
\end{align}

From \eqref{indet} and properties of $h$ we have for every $\xi\in\GG$ and every $t\in ]0,1[$
\begin{align*}
g(t\xi)=& h(t^2\det(I+\xi)+(1-t)^2+t(1-t){\rm tr}\left(I+\xi\right))\\
\le&h(t^2\det(I+\xi))\\
\le&\frac{1}{t^{2r}} h(\det(I+\xi))\\
=&\frac{1}{t^{2r}} h(\det(I+\xi))-h(\det(I+\xi))+g(\xi)\\
\le&\frac{1-t^{2r}}{t^{2r}}(1+g(\xi))+g(\xi)
\end{align*}
which implies $\Delta^1_g(t)\le \frac{1-t^{2r}}{t^{2r}}$. Letting $t\to 1$ we obtain~\ref{exa7}.
\end{proof}
\begin{remark} We may think that the function $\g$ could be convex but it is not the case in general, see Remark~\ref{noconvex}. 
\end{remark}
We define the function $G:\MM^{2\times 2}\to[0,\infty]$ by
\begin{equation}\label{Intro-Application-Eq1}
G(\xi):=\left\{
\begin{array}{ll}
|\xi|^p+g(\xi)&\hbox{if }\xi\in\GG\\
\infty&\hbox{otherwise.}
\end{array}
\right.
\end{equation}
Using~\ref{exa1}, Proposition~\ref{gpoly}~\ref{exa5} and~\ref{exa6} it is easy to see 
\begin{lemma} The function $G$ defined in \eqref{Intro-Application-Eq1} satisfies \ref{A2}, \ref{A3} and \ref{A0}.
\end{lemma}
Let $\W:\RR^2\times\MM^{2\times 2}\to[0,\infty]$ be defined by
\begin{equation}\label{Appli-d=2-Eq1}
\W(x,\xi):=\left\{
\begin{array}{ll}
\Phi(x,\xi)+g(\xi)&\hbox{if }\xi\in\GG\\
\infty&\hbox{otherwise,}
\end{array}
\right.
\end{equation}
where $\Phi:\RR^2\times\MM^{2\times 2}\to[0,\infty]$ is a quasiconvex function, $1$-periodic with respect to its first variable and of $p$-polynomial growth, i.e., there exist $c,C>0$ such that
\begin{equation}\label{Intro-Application-p-polynomial-Growth}
c|\xi|^p\leq \Phi(x,\xi)\leq C(1+|\xi|^p)
\end{equation}
for all $(x,\xi)\in\RR^2\times\MM^{2\times 2}$. 

The following proposition shows that such a $\W$ is consistent with the assumptions of Theorem~\ref{main result-cor} as well as with the two basic conditions of hyperelasticity, i.e., the non-interpenetration of the matter and the necessity of an infinite amount of energy to compress a finite volume of matter into zero volume.

\begin{proposition}\label{propW}We have 
\begin{enumerate}[label={(\roman{*})}]
\setcounter{enumi}{0}
\item $\W$ is $1$-periodic with respect to the first variable;
\item $\W$ satisfies~\ref{B1} and~\ref{A1} with $G$ given by \eqref{Intro-Application-Eq1};
\item for every $(x,\xi)\in\RR^d\times\GG$, $\W(x,\xi)<\infty$ if and only if $\det(I+\xi)>0$;
\item for every $x\in\RR^d$, $\W(x,\xi)\to\infty$ as $\det(I+\xi)\to0$ if $\lim\limits_{x\to 0}h(x)=\infty$;
\item\label{exa15} $\W$ is periodically ru-usc.
\end{enumerate}
\end{proposition} 
\begin{proof}The only not direct property is~\ref{exa15}. Fix any $t\in[0,1]$, any $x\in\RR^d$ and any $\xi\in\GG$. As $\Phi$ is quasiconvex and satisfies \eqref{Intro-Application-p-polynomial-Growth}, then repeating the same arguments as in the proof of Theorem~\ref{classic-homogenization} (Subsect.~\ref{pgrowth-homogenization}) we see that $\Phi$ is periodically ru-usc with $a\equiv 1$.

On the other hand, as $g$ is ru-usc by Proposition~\ref{gpoly} ~\ref{exa7} we have
\begin{align}\label{Intro-Eq-Prop-Eq2}
g(t\xi)-g(\xi)\leq \Delta_g^1(t)(1+g(\xi)).
\end{align}
From \eqref{Intro-Eq-Prop-Eq1} and \eqref{Intro-Eq-Prop-Eq2} we deduce that 
\[
\W(x,t\xi)-\W(x,\xi)\leq \max\left\{\Delta_\Phi^2(t),\Delta_g^1(t)\right\}(2+\W(x,\xi)).
\]
Passing to the supremum on $x$ and $\xi$, we obtain
\[
\sup_{x\in\RR^d}\sup_{\xi\in\GG}{\W(x,t\xi)-\W(x,\xi)\over 2+W(x,\xi)}\le \max\left\{\Delta_\Phi^2(t),\Delta_g^1(t)\right\},
\]
and~\ref{exa15} follows when $t\to1$.
\end{proof}
\subsection*{Concrete example} 
Every $\xi\in\MM^{2\times 2}$ is denoted by
\begin{align*}
\xi:=\begin{pmatrix}
\xi_{11}&\xi_{12} \\
\xi_{21}&\xi_{22}
\end{pmatrix}
\end{align*}

Define the set
\begin{align*}
\GG:=\Big\{\xi\in\MM^{2\times 2}:\min\left\{1+\xi_{11},1+\xi_{22}\right\}\ssup \max\left\{\lvert\xi_{12}\rvert,\lvert\xi_{21}\rvert\right\}\Big\},
\end{align*}
Property~\ref{exa1} is evident. The subset $\GG$ is open and convex as intersection of open convex sets, so~\ref{exa2} holds. The assertion~\ref{exa3} is satisfied because for every $\xi\in\GG$ we have
\begin{align*}
\det(I+\xi)=\left(1+\xi_{11}\right)\left(1+\xi_{22}\right)-\xi_{12}\xi_{21}\ssup  \lvert\xi_{12}\xi_{21}\rvert-\xi_{12}\xi_{21}\ge 0.
\end{align*}
To verify~\ref{exa4}, we note that for every $\xi,\zeta\in\GG$ it holds
\begin{align*}
&{\rm tr}\left({\rm cof}(I+\xi)^\intercal(I+\zeta)\right)\\
=&\left(1+\xi_{11}\right)\left(1+\zeta_{11}\right)+\left(1+\xi_{22}\right)\left(1+\zeta_{11}\right)-\xi_{12}\zeta_{21}-\xi_{21}\zeta_{12}\\
>&\lvert\xi_{12}\zeta_{21}\rvert+\lvert\xi_{21}\zeta_{12}\rvert-\xi_{12}\zeta_{21}-\xi_{21}\zeta_{12}\ge 0.
\end{align*}
We can take $g:\MM^{2\times 2}\to [0,\infty]$ defined by
\[
g(\xi):=\left\{
\begin{array}{cl}
\displaystyle h(\det(I+\xi)) &\mbox{ if }\xi\in \GG\\ \\
\infty&\mbox{ otherwise}
\end{array}
\right.
\]
where $h$ is the convex and nonincreasing function defined by $h(x):=\frac1x$ for all $x\ssup 0$ satisfying \eqref{h} with $r=1$.
\begin{remark}\label{noconvex} It is easy to see that $\g(\cdot):=g(\cdot-I)$ is polyconvex but not convex. Indeed, consider $F\in I+\GG$ defined by
\begin{align*}
F:=\begin{pmatrix}
1&\frac12 \\ \\
-\frac12&1
\end{pmatrix}
\end{align*}
then $\g(\frac12F+ \frac12F^\intercal)=1\ssup \frac45= \frac12\left(\g(F)+\g(F^\intercal)\right)$.
\end{remark}
\begin{remark}\label{nofrind} A necessary condition for $\g$ to be frame indifferent is that $P(I+\GG)\subseteq I+\GG$ for all $P\in {\rm SO}(2)$, which in particular means that ${\rm SO}(2)\subseteq I+\GG$ since~\ref{exa1}. But, this is not true because the rotation of angle $\frac\pi2$ does not belong to $I+\GG$.  
\end{remark}

\section*{Appendix}

\setcounter{section}{1}
\subsection{Dacorogna-Acerbi-Fusco relaxation theorem}\label{DAF}
Let $\Omega\subset\RR^d$ be an open bounded set with Lipschitz boundary. Let $p\in [1,\infty[$. Let $f:\Omega\times\MM^{m\times d}\to [0,\infty]$ be a Borel measurable integrand with $p$-polynomial growth, i.e., there exist $c,C\ssup 0$ such that for every $\xi\in\MM^{m\times d}$ it holds
\begin{align*}
c\lvert \xi\rvert^p\le f(x,\xi)\le C(1+\lvert\xi\rvert^p).
\end{align*}
Let $J:W^{1,p}(\Omega;\RR^m)\to [0,\infty]$ be defined by 
\begin{align*}
J(u):=\int_\Omega f(x,\nabla u(x))dx.
\end{align*} 
Then for every $u\in W^{1,p}(\Omega;\RR^m)$
\[
\overline{J}(u):=\inf\left\{\liminf_{\eps\to 0}J(u_\eps): u_\eps\to u \mbox{ in }L^p(\Omega;\RR^m)\right\}=\int_\Omega \Q f(x,\nabla u(x))dx,
\]
where $\Q f:\Omega\times \MM^{m\times d}\to [0,\infty[$ is the {\em quasiconvexification} of $f$ given by the Dacorogna formula
\begin{align*}
\Q f(x,\xi)=\inf\left\{\int_Y f(x,\xi+\nabla \varphi(y))dy:\varphi\in W^{1,\infty}_0(Y;\RR^m)\right\}.
\end{align*}
As a consequence we have 
\begin{flalign*}
&\inf\left\{\int_\Omega f(x,\nabla u)dx:u\in W^{1,p}(\Omega;\RR^m)\right\}
=\inf\left\{\int_\Omega \Q f(x,\nabla u)dx:u\in W^{1,p}(\Omega;\RR^m)\right\}.&
\end{flalign*}
\subsection{Alexandrov theorem}\label{alexandrov} 
If the sequence $\{\mu_\eps\}_{\eps>0}$ of nonnegative finite Radon measures on $\Omega$ weakly converges to the Radon measure $\mu$, i.e.,
\[
\lim_{\eps\to 0}\int_\Omega\phi d\mu_\eps=\int_\Omega\phi d\mu\hbox{ for all } \phi\in C_{\rm c}(\Omega),
\]
then
\begin{enumerate}[label=(\alph*)]
\item $\liminf\limits_{\eps\to 0}\mu_\eps(U)\geq\mu(U)\hbox{ for all open sets } U\subset\Omega;$

\item $\limsup\limits_{\eps\to 0}\mu_\eps(K)\leq\mu(K)\hbox{ for all compact sets } K\subset\Omega;$

\item $\lim\limits_{\eps\to 0}\mu_\eps(B)=\mu(B)$ for all bounded Borel sets $B\subset\Omega$ with $\mu(\partial B)=0$.
\end{enumerate}

\bigskip

\bigskip

\subsubsection*{Acknowledgement} The first author acknowledges the mathematics department of the University of Tunis El Manar for the invitation in may 2012, where this work was initiated.

\end{document}